%% file: main.tex
\documentclass[11pt]{article}

\usepackage[margin=1in]{geometry}
\usepackage[T1]{fontenc}
\usepackage[utf8]{inputenc}
\usepackage{amsmath,amssymb,amsthm,mathrsfs}
\usepackage{graphicx}
\usepackage{url}

\newtheorem{thm}{Theorem}
\newtheorem{lem}[thm]{Lemma}
\newtheorem{prop}[thm]{Proposition}
\newtheorem{cor}[thm]{Corollary}
\theoremstyle{remark}
\newtheorem{rem}[thm]{Remark}

\usepackage[hidelinks]{hyperref}
\hypersetup{
  pdftitle={Exact Decision and a Surjective Stabilizer Atlas for Affine Coprime-Factor Segments},
  pdfauthor={Junkai Qiu},
  pdfsubject={Simultaneous stabilization and rational matrix interpolation},
  pdfkeywords={simultaneous stabilization, affine coprime-factor segment, exact decision, double-Cayley factorization, surjective atlas}
}

\begin{document}

\title{Exact Decision and a Surjective Stabilizer Atlas for Affine Coprime-Factor Segments}

\author{Junkai Qiu\\
School of Mathematical Sciences\\
Dalian University of Technology\\
Dalian 116024, China\\
\texttt{qjk@mail.dlut.edu.cn}}

\date{}

\maketitle

\begin{abstract}

We consider simultaneous internal stabilization of square proper
real-rational plants of arbitrary but fixed finite input-output dimension
$n$, arranged as an affine right-coprime factor segment, under explicit
nondegeneracy hypotheses.  After a classical normalization, the remaining
global spectral-cut constraint and the reconstruction constraints at finite
poles and at infinity are encoded, without loss, as a finite-dimensional
semialgebraic seed, using a function-level double-Cayley factorization.
For input coefficients in an effectively presented real closed field,
that seed yields an exact existence decision that does not enumerate
controller McMillan degree; the decision engine is classical quantifier
elimination.  The same encoding, now ranging over rational radii, admissible
seeds, and residual unimodular and rational Schur parameters, yields a
surjective atlas of all proper real-rational common stabilizers of the
segment.  The results concern this
structured class and are compatible with the rational undecidability of
simultaneous stabilization of three general plants.
\end{abstract}



\section{Introduction}

This paper studies one structured simultaneous-stabilization problem: a
continuous affine segment of square proper real-rational plants of
arbitrary but fixed finite input-output dimension $n$, represented by
stable right-coprime factor pairs.  Under explicit nondegeneracy
hypotheses, the questions are whether a proper real-rational common
internal stabilizer exists and how all such controllers may be described
without an a priori bound on McMillan degree.

Coprime-factor methods for multivariable plants are classical
\cite{Youla1974single}.  Residual-unimodularity parameterizations of
simultaneous stabilizers are given in
\cite{Saeks1982fractional,Vidyasagar1982algebraic}.
Generic simultaneous pole-placement results are available as well
\cite{GhoshByrnes1983simultaneous}.
Characterizations of the full stabilizer set are available, yet the free
parameter remains subject to a unimodularity constraint that is difficult
to resolve constructively \cite{Obinata1988characterization}.  The present
paper does not invent that parameterization; it asks when the residual
unimodular constraint can be resolved on this structured continuous
family.  Special plant classes yield useful positive constructions
\cite{Chen1995simultaneous,Gundes2001simultaneously}.  For three general
plants, simultaneous stabilization is rationally undecidable
\cite{BlondelGevers1993undecidable}, related semialgebraic limitations are
known \cite{Bertilsson1995semialgebraic}, and conditions on the positive
real axis do not suffice \cite{Blondel1994simultaneous}.
None of these facts is claimed as a contribution below.

A second line treats affine and polytopic families by spectral cuts and
interpolation on factor segments.  Finite vertex reductions and sufficient
conditions are known for affine characteristic-polynomial families and for
polytopic plants
\cite{Djaferis1991stabilize,Fonte2019stabilization,Fonte2025sufficient},
and avoidance of the negative real axis already appears for
single-input/single-output (SISO) segments \cite{Fonte2008conditions}.
Analytic interpolation proceeds from Ghosh \cite{Ghosh1986transcendental}
to a parameterized SISO construction that includes derivative data
\cite{Cui2023parameterized}.  The closest recent multi-input/multi-output
(MIMO) treatments of affine factor segments
\cite{Cui2024multivariable,Cui2024general} work under simple zeros,
rank-one adjugate data, and a single matrix-square-root interpolation.
Their parameterization describes all solutions of that interpolation
problem; it does not automatically describe all proper real-rational
common stabilizers of the segment, and it does not treat the confluent
Smith jets generated by repeated zeros.

The gap addressed here is how to treat simultaneously the global
spectral-cut constraint, the complete multiple-zero reconstruction data,
properness at infinity, and real-rational solutions of arbitrary degree,
without sacrificing existence or completeness of the solution set.
Pointwise accretive products, matrix square roots, commuting Lyapunov
metrics, rational spectral factors, and confluent matrix interpolation are
classical instruments
\cite{Ballantine1975accretive,Wu1989operator,Higham2008functions,Narendra1994common,Baggio2016factorization,Ball2018bitangential};
they are used below and are not claimed as contributions.  The increment
is their function-level assembly into a lossless finite seed for the
stated segment, and a covering of every proper real-rational common
stabilizer that the seed determines.

Related work on exact stabilization certificates uses numerical algebraic
geometry \cite{Menini2025exact}.  Strongly $F$-positive real interpolation
\cite{Stefanovski2023interpolation} and bistable tangential interpolation
\cite{Stefanovski2026bistable} address strong stabilization, which also
requires the controller itself to be stable.  These results provide related
algebraic and interpolation constructions.  The question here is the
lossless finite decision and full stabilizer set for the stated affine
factor segment, without imposing stability on the controller itself.

The paper makes three contributions.

\begin{enumerate}
\item \emph{Function-level double-Cayley factorization.}
A stable real-rational matrix function avoids the spectral cut
$(-\infty,0]$ throughout the closed stable region if and only if, up to
unimodular similarity over the stable rational ring, it is the product of
two Cayley transforms of strict rational Schur functions.  The increment
is the global identity that preserves real rationality, stability,
stably invertible similarity, and consistency on the whole closed stable
region, rather than a pointwise matrix factorization.

\item \emph{Lossless finite seed and exact decision.}
For a nondegenerate affine right-coprime factor segment of size $n$, the
spectral-cut constraint and the complete reconstruction
constraints---including repeated Smith indices, one additional Taylor
coefficient at infinity, and the data needed for a unimodular Hermite
lift---are encoded without loss as a finite-dimensional semialgebraic
seed.  Every admissible controller, of any McMillan degree, compresses to
a seed of this format, and every acceptable seed lifts to a legal
controller.  Over an effectively presented real closed coefficient field,
classical quantifier elimination applied to the seed yields a terminating
exact existence decision that does not enumerate controller McMillan
degree.  After a positive sample, certain auxiliary constructions in the
lift remain terminating enumerations; they are not a search over the
McMillan degree of the controller.

\item \emph{Surjective atlas of all common stabilizers.}
Ranging over admissible seeds of that encoding, a surjective atlas
produces every proper real-rational common stabilizer of the segment.
The free parameters include the rational radius and admissible seed, the
residual unimodular Hermite freedom, not merely a neighborhood of the
identity, together with rational Schur interpolants.  The atlas is
surjective and need not be unique; it is a constraint-resolved covering of
the stated segment, not an unconstrained Youla parameter.
\end{enumerate}

An explicit example shows that the subclass of common stabilizers whose
normalized matrix admits a strictly accretive real-rational square root is
nonempty and strictly smaller than the full set.  That comparison is
recorded as evidence that two Cayley factors are used for completeness,
not as a further contribution, and it does not assert that a square-root
interpolation is empty on the same instance.

Throughout, $n$ is the input-output dimension of the square plants, not
the McMillan degree of the plants or of the controller, and affine
interpolation of coprime factors is not affine interpolation of transfer
matrices.  The theorems do not address general simultaneous
stabilization.  Blondel and Gevers established that simultaneous
stabilization of three general plants is rationally undecidable
\cite{BlondelGevers1993undecidable}; the present results are compatible
with that boundary by restriction of the plant family, not by a claim of
tractability for unrestricted instances.

\section{Problem formulation and the inherited reduction}

\subsection{The plant class and standing hypotheses}

Throughout, $I_n$ denotes the $n\times n$ identity matrix, $X^*$ the conjugate
transpose of $X$, $X^{-*}=(X^{-1})^*$, and $\sigma(X)$ the spectrum of $X$.
Inequalities between Hermitian matrices are understood in the Loewner order.
Put $\overline{\mathbb C}_+^{\,e}
=\{s\in\mathbb C:\operatorname{Re}s\geq0\}\cup\{\infty\}$, and let
$\mathcal H=\mathcal{RH}_\infty$ denote the ring of proper
real-rational functions analytic on $\overline{\mathbb C}_+^{\,e}$.  A square
matrix $F$ over $\mathcal H$ is \emph{unimodular} if its inverse also has
entries in $\mathcal H$; the unimodular group is denoted by
$\mathrm{GL}_n(\mathcal H)$.  Equivalently, $F\in\mathrm{GL}_n(\mathcal H)$ if
and only if $F\in\mathcal H^{n\times n}$ and $\det F$ has no zeros on
$\overline{\mathbb C}_+^{\,e}$.

Let $n\geq1$ be a fixed finite integer.  In this paper $n$ is the
input-output size of the plants, not their McMillan degree: the latter
remains unbounded a priori.  A pair
$(N,D)\in\mathcal H^{n\times n}\times\mathcal H^{n\times n}$ is
\emph{right coprime} over $\mathcal H$ if there exist
$X,Y\in\mathcal H^{n\times n}$ with $XN+YD=I_n$.  Consider stable
right-coprime pairs $(N_i,D_i)\in\mathcal H^{n\times n}\times
\mathcal H^{n\times n}$, $i=0,1$, and define
\begin{equation}
\begin{split}
N_\lambda&=(1-\lambda)N_0+\lambda N_1,\\
D_\lambda&=(1-\lambda)D_0+\lambda D_1,\\
P_\lambda&=N_\lambda D_\lambda^{-1},\qquad 0\leq\lambda\leq1.
\end{split}
\label{eq:plant-line}
\end{equation}
The interpolation is affine in the factors, not in the plants:
$P_\lambda$ need not equal $(1-\lambda)P_0+\lambda P_1$.

We recall three standard facts used throughout.  They are inherited and are
not contributions of the paper.
\begin{enumerate}
\item \emph{Internal stability.}
With the negative-feedback convention, a proper pair $(P,K)$ with
right-coprime plant factors $P=ND^{-1}$ and left-coprime controller factors
$K=D_c^{-1}N_c$ is internally stable if and only if
$\Delta=N_cN+D_cD$ lies in $\mathrm{GL}_n(\mathcal H)$
\cite{Vidyasagar1982algebraic}.
\item \emph{Properness of a left-coprime controller.}
If $K=D_c^{-1}N_c$ is a stable left-coprime factorization, then $K$ is
proper if and only if $D_c(\infty)$ is nonsingular
\cite{Vidyasagar1982algebraic}.
\item \emph{Constant spectral cut.}
For a constant matrix $A_0$, the pencil $(1-\lambda)I_n+\lambda A_0$ is
invertible for every $\lambda\in[0,1]$ if and only if
$\sigma(A_0)\cap(-\infty,0]=\emptyset$.
\end{enumerate}

The standing hypotheses are that every $P_\lambda$ is proper, every
$(N_\lambda,D_\lambda)$ is right coprime over $\mathcal H$, and
\begin{equation}
\mathcal M=
\begin{bmatrix}N_0&N_1\\D_0&D_1\end{bmatrix},
\qquad \det\mathcal M\not\equiv0.
\label{eq:M}
\end{equation}
The last condition means that $\mathcal M$ is invertible over the
real-rational function field; it does not require $\mathcal M^{-1}$ to have
entries in $\mathcal H$, and it does not prohibit zeros of $\det\mathcal M$
of any finite multiplicity in $\overline{\mathbb C}_+^{\,e}$.  Properness and
right coprimeness are assumed on the whole segment: they do not follow
automatically from the corresponding properties at the two endpoints.

\subsection{Normalized reduction}

Write a proper controller in a stable left-coprime form
$K=D_c^{-1}N_c$, where $N_c,D_c\in\mathcal H^{n\times n}$ and
$(D_c,N_c)$ is left coprime over $\mathcal H$.  A \emph{common stabilizer}
of \eqref{eq:plant-line} is a single proper real-rational controller that
internally stabilizes every $P_\lambda$; no stability requirement is imposed
on $K$ itself.  The standard coprime-factor criterion says that $K$
internally stabilizes $P_\lambda$ exactly when
\begin{equation*}
\Delta_\lambda=N_cN_\lambda+D_cD_\lambda
\in \mathrm{GL}_n(\mathcal H).
\end{equation*}
If $K$ stabilizes the segment, then $\Delta_0\in\mathrm{GL}_n(\mathcal H)$.
Left multiplication of $(N_c,D_c)$ by $\Delta_0^{-1}$ preserves $K$ and
normalizes $\Delta_0=I_n$.  With $A=\Delta_1$ this gives
\begin{equation}
[\,N_A\;D_A\,]=[\,I_n\;A\,]\mathcal M^{-1},
\qquad
\Delta_\lambda=(1-\lambda)I_n+\lambda A.
\label{eq:reconstruct}
\end{equation}
The inverse of $\mathcal M$ exists as a real-rational matrix by \eqref{eq:M};
the blocks $N_A,D_A\in\mathbb R(s)^{n\times n}$ are the unique solution of
$[\,N_A\;D_A\,]\mathcal M=[\,I_n\;A\,]$.

Define
\begin{equation}
\begin{split}
\mathfrak A_{\mathrm{adm}}^{(n)}=\{A\in\mathcal H^{n\times n}:\;&
[\,I_n\;A\,]\mathcal M^{-1}\in\mathcal H^{n\times 2n},\\
&\sigma(A(s))\cap(-\infty,0]=\emptyset
\text{ for all }s\in\overline{\mathbb C}_+^{\,e},\\
&\det D_A(\infty)\neq0\}.
\end{split}
\label{eq:Aadm}
\end{equation}
The spectral-cut condition is strictly weaker than pointwise strict
accretivity $A(s)+A(s)^*\succ0$, and it is not a condition on real
frequencies alone.

\begin{prop}[Normalized reduction]\label{prop:reduction}
There exists a proper real-rational controller simultaneously stabilizing
\eqref{eq:plant-line} if and only if $\mathfrak A_{\mathrm{adm}}^{(n)}$ is
nonempty.  Moreover,
\begin{equation}
K=D_A^{-1}N_A,
\qquad [\,N_A\;D_A\,]=[\,I_n\;A\,]\mathcal M^{-1},
\label{eq:KfromA}
\end{equation}
maps $\mathfrak A_{\mathrm{adm}}^{(n)}$ onto all proper real-rational common
stabilizers.  For each such controller $K$, the matrix $A$ is uniquely
determined: the normalization $\Delta_0=I_n$ removes the left unimodular
freedom of the controller factorization.  This uniqueness of $A$ on the
present plant class is not to be confused with uniqueness of the atlas
constructed in Section~\ref{sec:atlas}.
\end{prop}

\begin{proof}
Invoke the coprime-factor criterion
\cite{Vidyasagar1982algebraic}, the left-coprime
properness test that $D_c(\infty)$ is nonsingular
\cite{Vidyasagar1982algebraic}, and the constant
spectral-cut computation recalled before \eqref{eq:M}.  If
$A\in\mathfrak A_{\mathrm{adm}}^{(n)}$, then $D_A(\infty)$ is nonsingular, so
\eqref{eq:KfromA} is proper, while $N_AN_0+D_AD_0=I_n$ shows that the
reconstructed row is left coprime over $\mathcal H$.  Equation
\eqref{eq:reconstruct} holds.  For $0<\lambda\leq1$, a singularity of
$\Delta_\lambda(s)$ would place $-(1-\lambda)/\lambda$ in $\sigma(A(s))$,
contradicting the spectral-cut condition in \eqref{eq:Aadm}.  The case
$\lambda=0$ is $\Delta_0=I_n$.  Hence every $\Delta_\lambda$ lies in
$\mathrm{GL}_n(\mathcal H)$.

Conversely, left multiplication by $\Delta_0^{-1}$ preserves any common
stabilizer and normalizes $\Delta_0=I_n$.  Setting $A=\Delta_0^{-1}\Delta_1$
yields \eqref{eq:reconstruct} and, by the same three facts, places this $A$
in $\mathfrak A_{\mathrm{adm}}^{(n)}$.  Thus the map is onto.

The left unimodular freedom of a coprime controller factorization is killed
by the normalization $\Delta_0=I_n$.  Explicitly, $N_A=D_AK$ and
$D_A(KN_0+D_0)=I_n$ give
\[
A=(KN_0+D_0)^{-1}(KN_1+D_1),
\]
which depends only on $K$ and the fixed factors $(N_i,D_i)$.
\end{proof}

\subsection{Two constraints and a running example}

After Proposition~\ref{prop:reduction}, two constraints of different
character remain on $A\in\mathfrak A_{\mathrm{adm}}^{(n)}$.
\begin{enumerate}
\item \emph{Global spectral cut.}
The values of $A$ must avoid the nonpositive real axis throughout the closed
stability domain $\overline{\mathbb C}_+^{\,e}$.  This is a condition on the
range of $A$, not on its poles.  Section~\ref{sec:cayley} gives a global
rational factorization of this constraint.
\item \emph{Local reconstruction poles.}
The product $[\,I_n\;A\,]\mathcal M^{-1}$ must remain in
$\mathcal H^{n\times 2n}$, and $D_A(\infty)$ must be nonsingular.  Because
$\mathcal M^{-1}$ is merely rational, $A$ must cancel the poles of
$\mathcal M^{-1}$, including the point at infinity.
Section~\ref{sec:smith} encodes these conditions by finite Smith--Hermite
data.
\end{enumerate}
Neither constraint implies the other.  The continuous parameter $\lambda$ has
disappeared from the stability condition on the unknown, but the reduction
itself imposes no a priori bound on the McMillan degree of $A$ or of $K$.
The two constraints meet in a finite seed in
Section~\ref{sec:decision}.

To keep the subsequent constructions concrete, we record here the $2\times2$
running example used throughout the paper, as the case $n=2$ of
\eqref{eq:plant-line}.  Write $z=(1-s)/(1+s)$, which maps
$\overline{\mathbb C}_+^{\,e}$ onto the closed unit disk and sends
$s=\infty$ to $z=-1$.  In that coordinate, set
\begin{equation*}
\begin{split}
p(z)&=\frac{z^2(1+z)}8,\qquad
q(z)=-\frac{1+3z}8,\\
h(z)&=\frac{(1+z)^2}{64},\qquad
G(z)=\begin{bmatrix}p(z)&h(z)\\0&q(z)\end{bmatrix},\\
C(z)&=-4(1+z)I_2.
\end{split}
\end{equation*}
Take
\begin{equation*}
(N_0,D_0)=(0,I_2),\qquad
(N_1,D_1)=(G,I_2-CG).
\end{equation*}
For every $\lambda\in[0,1]$ one has $CN_\lambda+D_\lambda=I_2$, so the affine
factor pair remains right coprime, while $C(-1)=0$ forces every plant to be
proper.  A direct computation gives
$\det\mathcal M=\det G=-z^2(1+z)(1+3z)/64$.  In particular, $\det G$ has a
double zero at $z=0$ and a simple zero at $z=-1$, the latter representing
the point at infinity.  These zeros are the source of the local
reconstruction constraints; they are not treated in the present section.  A
complete exact certificate is deferred to Section~\ref{sec:examples}.

\section{Global rational factorization of the spectral-cut constraint}
\label{sec:cayley}

Proposition~\ref{prop:reduction} leaves two constraints of different type on
the normalized matrix $A$: a global spectral-cut condition on
$\overline{\mathbb C}_+^{\,e}$, and local reconstruction conditions at the
poles of an endpoint block inverse.  The present section treats the first
constraint, for every finite square size $n\geq1$, as a factorization problem
over the stable rational ring.  Only the spectral-cut condition is used here.

Use $z=(1-s)/(1+s)$ to map $\overline{\mathbb C}_+^{\,e}$ onto the closed unit
disk $\overline{\mathbb D}$, where $\mathbb D=\{z\in\mathbb C:\lvert z\rvert<1\}$;
the point $s=\infty$ corresponds to $z=-1$.  For a continuous-time transfer
matrix $F$, write $\widehat F(z)=F((1-z)/(1+z))$.  For a disk-domain rational
matrix $G$, write $G^\sim(z)=G(1/\bar z)^*$.  Let
$\mathcal R=\mathcal{RH}_\infty(\mathbb D)$ be the ring of real-rational
functions analytic on a neighborhood of $\overline{\mathbb D}$.  The bilinear
change of variable is a ring isomorphism $\mathcal H\cong\mathcal R$ that
preserves unimodularity.  Henceforth a matrix already viewed as an element of
$\mathcal R^{n\times n}$ is written without a hat.  The unimodular group
$\mathrm{GL}_n(\mathcal R)$ consists of those matrices in $\mathcal R^{n\times n}$
whose inverses also lie in $\mathcal R^{n\times n}$.  The norm
$\lVert\cdot\rVert_\infty$ is the matrix $H^\infty$ norm induced by the
Euclidean operator norm.  A square matrix $P$ is \emph{positive stable} if every
eigenvalue of $P$ has strictly positive real part, and \emph{strictly accretive}
if $P+P^*>0$.  Strict accretivity implies positive stability, but the converse
fails: the factorization below therefore inserts one unimodular similarity that
converts two commuting positive-stable factors into strictly accretive ones.

Define the Cayley transform by
\begin{equation}
\mathcal C(Q)=(I_n+Q)(I_n-Q)^{-1}.
\label{eq:car}
\end{equation}
A rational matrix $Q$ is Schur when $\lVert Q\rVert_\infty\leq1$ and strict
Schur when $\lVert Q\rVert_\infty<1$.  The elementary identity
\[
\mathcal C(Q)+\mathcal C(Q)^*
=2(I_n-Q^*)^{-1}(I_n-Q^*Q)(I_n-Q)^{-1}
\]
shows that $\mathcal C$ is a bijection between the strict rational Schur class
and the class of pointwise strictly accretive matrices in
$\mathcal R^{n\times n}$, with inverse
$Q=(P-I_n)(P+I_n)^{-1}$.  We use this correspondence as a standard identity,
not as a new lemma.

\begin{lem}[Standard accretive product]\label{lem:accretive-product}
If $P+P^*>0$ and $Q+Q^*>0$, then
$\sigma(PQ)\cap(-\infty,0]=\emptyset$
\cite{Ballantine1975accretive,Wu1989operator}.
The argument does not require $P$ and $Q$ to commute.
\end{lem}

\begin{proof}
See \cite{Ballantine1975accretive,Wu1989operator}.
\end{proof}

Write $\mathcal A_{\mathrm{cut}}^{(n)}$ for the class of all
$A\in\mathcal R^{n\times n}$ such that
$\sigma(A(z))\cap(-\infty,0]=\emptyset$ whenever $\lvert z\rvert\leq1$.
After the disk coordinate change, this is precisely the spectral-cut part of
$\mathfrak A_{\mathrm{adm}}^{(n)}$.

\begin{thm}[Double-Cayley factorization]\label{thm:double-cayley}
For every $n\geq1$,
\begin{equation}
\begin{split}
\mathcal A_{\mathrm{cut}}^{(n)}
&=\{U^{-1}\mathcal C(Q_1)\mathcal C(Q_2)U:\\
&\hspace{12mm}U\in\mathrm{GL}_n(\mathcal R),\\
&\hspace{12mm}\lVert Q_1\rVert_\infty<1,\
\lVert Q_2\rVert_\infty<1\}.
\end{split}
\label{eq:double-cayley}
\end{equation}
\end{thm}

\begin{proof}
The reverse inclusion is classical matrix algebra: if
$\lVert Q_k\rVert_\infty<1$, the Cayley identity after \eqref{eq:car} makes
each $\mathcal C(Q_k)$ strictly accretive, Lemma~\ref{lem:accretive-product}
keeps the product off the cut, and similarity preserves spectra.  No
restriction on $n$ enters.

The forward inclusion is the only original step.  Write $X=p(A)$ and
$Y=X^{-1}A$ by approximating the principal square root of the spectral
image of $A$ (Mergelyan \cite{Rudin1987real}; the square root itself need
not be rational \cite{Higham2008functions}).  The two factors commute and
are positive stable in $\mathcal R^{n\times n}$.  A common Lyapunov metric
$H$ for the pair is obtained by applying the commuting constant-matrix
construction of \cite{Narendra1994common} fiberwise on the circle, and is
real-rational by vectorization of the Lyapunov equation.  A rational spectral
factor $H=U^\sim U$ with $U,U^{-1}\in\mathcal R^{n\times n}$
\cite{Baggio2016factorization} then makes both $UXU^{-1}$ and $UYU^{-1}$
strictly accretive, so the inverse Cayley map yields the two strict Schur
factors.  The same $U$ serves both factors; they need not be equal.  None
of these citations depends on $n=2$.  The estimates that make $H$
rational and the two similar copies simultaneously accretive are recorded
in Appendix~\ref{app:double-cayley}.  The novelty is this simultaneous
real-rational, stable, and unimodular consistency, not the pointwise
accretive product.
\end{proof}

\begin{rem}
The two Schur factors in \eqref{eq:double-cayley} need not be equal, and the
sufficient inclusion does not require them to commute.  Consequently,
Theorem~\ref{thm:double-cayley} does not require a real-rational square root of
$A$.  A full example separating the present factorization from the single
strictly accretive real-rational square-root class is deferred to
Section~\ref{sec:examples}.
\end{rem}

\section{Complete finite data for reconstruction}
\label{sec:smith}

Proposition~\ref{prop:reduction} leaves two constraints of different type on
the normalized matrix $A$.  Theorem~\ref{thm:double-cayley} factorizes the
global spectral-cut condition.  The present section answers a different
question: before any global unimodular or Schur lift is attempted, which
finite data must be recorded so that the reconstruction
\eqref{eq:reconstruct} is analytic on the closed disk and the controller
$K=D_A^{-1}N_A$ is proper?  Only those local conditions are treated here.

Write $\widehat{\mathcal M}$ for the disk form of \eqref{eq:M}, and write
$\widehat A$ for the disk form of $A$.  Stability of $\widehat{\mathcal M}$
does not make its inverse stable.  The only candidate poles of
$[\,I_n\;\widehat A\,]\widehat{\mathcal M}^{-1}$ in $\overline{\mathbb D}$
are the zeros of $\det\widehat{\mathcal M}$, and properness is a condition
at $s=\infty$, which is the point $z=-1$.

\subsection{Nodes, local Smith chains, and pole cancellation}

The node set is
\begin{equation}
\mathcal Z=\{\zeta\in\overline{\mathbb D}:\det\widehat{\mathcal M}(\zeta)=0\}
\cup\{-1\}.
\label{eq:nodes}
\end{equation}
The point $-1$ is included only once.  The set is finite because a rational
determinant that is not identically zero has isolated zeros.  Boundary zeros
of $\det\widehat{\mathcal M}$ must be cancelled on the same footing as
interior zeros; it is not enough to retain only the zeros in the open disk.

At each distinct $\zeta_j\in\mathcal Z$, put $t=z-\zeta_j$.  The local Smith
form of an analytic matrix of full generic rank supplies analytic $L_j$ and
$R_j$, both invertible at $t=0$, such that
\begin{equation}
L_j(t)\widehat{\mathcal M}(\zeta_j+t)R_j(t)
=\operatorname{diag}(t^{\kappa_{j1}},\ldots,t^{\kappa_{j,2n}}),
\label{eq:smith}
\end{equation}
with $0\leq\kappa_{j1}\leq\cdots\leq\kappa_{j,2n}$.  Here
$\widehat{\mathcal M}$, $L_j$, and $R_j$ are $2n\times2n$.  Split the
$\ell$th column of $R_j$ into two $n\times1$ blocks,
\begin{equation}
R_j(t)e_\ell=
\begin{bmatrix}\nu_{j\ell}(t)\\v_{j\ell}(t)\end{bmatrix},
\qquad \nu_{j\ell},v_{j\ell}\in\mathbb C\{t\}^{n\times1},
\label{eq:root-chain}
\end{equation}
and write
$c_{j\ell}(t)=\nu_{j\ell}(t)+\widehat A(\zeta_j+t)v_{j\ell}(t)$.
The integer $\kappa_{j\ell}$ is the order that this root chain must cancel.
A single null vector, or a single adjugate value, does not encode the
conditions for $\kappa_{j\ell}>1$.

Let $[t^q]f(t)$ denote the coefficient of $t^q$ in the Taylor expansion of
$f$.  Then $[\,I_n\;\widehat A\,]\widehat{\mathcal M}^{-1}$ is analytic at
$\zeta_j$ if and only if
\begin{equation}
[t^q]\{\nu_{j\ell}(t)+\widehat A(\zeta_j+t)v_{j\ell}(t)\}=0,
\quad 0\leq q<\kappa_{j\ell},\quad 1\leq\ell\leq 2n.
\label{eq:smith-hermite}
\end{equation}
The local Smith form inverts as
$\widehat{\mathcal M}(\zeta_j+t)^{-1}
=R_j(t)\operatorname{diag}(t^{-\kappa_{j1}},\ldots,t^{-\kappa_{j,2n}})L_j(t)$,
and therefore
\begin{equation}
\begin{split}
&[\,I_n\;\widehat A(\zeta_j+t)\,]
\widehat{\mathcal M}(\zeta_j+t)^{-1}\\
&\qquad=[c_{j1}(t)\ \cdots\ c_{j,2n}(t)]\\
&\qquad\qquad\times\operatorname{diag}(t^{-\kappa_{j1}},\ldots,
t^{-\kappa_{j,2n}})L_j(t).
\end{split}
\label{eq:smith-product}
\end{equation}
Since $L_j$ is locally invertible, right multiplication by $L_j$ neither
creates nor removes poles, and it cannot cancel a genuine singularity by
mixing columns.  Analyticity of the product is therefore equivalent to
analyticity of each column $t^{-\kappa_{j\ell}}c_{j\ell}$, which is
\eqref{eq:smith-hermite}.  If every node condition holds, the reconstruction
row has no pole in the closed disk, hence lies in $\mathcal R^{n\times 2n}$.
Conversely, every stable reconstruction row satisfies all of the conditions.

The local Smith form of an analytic matrix is classical; it is used only to
read the $2n$ root chains of $\widehat{\mathcal M}$.  The increment is the
complete confluent encoding of pole cancellation for this affine
right-coprime factor segment, including repeated Smith indices, rank
degeneracies, and the point $z=-1$.  The equations
\eqref{eq:smith-hermite} do not assume simple zeros, a rank-one adjugate, or
pointwise diagonalizability
\cite{Cui2024multivariable,Cui2024general}.  The convolution form of
\eqref{eq:smith-hermite} is recorded in Appendix~\ref{app:smith}.

Nonreal zeros of a real-rational $\widehat{\mathcal M}$ occur in conjugate
pairs.  The local Smith data at $\bar\zeta$ may be taken as the
coefficientwise conjugate of the data at $\zeta$.  If $\widehat A$ is
real-rational as well, its jets are likewise conjugate, so each conjugate
pair of complex equations is equivalent to a real-and-imaginary pair.
Those conjugate constraints cannot be dropped by treating a nonreal node as
a real variable.

\subsection{The extra order at infinity}

At $\zeta_\infty=-1$, retain one additional jet coefficient after
\eqref{eq:smith-hermite} has been imposed.  Write
$c_{\infty,\ell}(t)=\nu_{\infty,\ell}(t)+\widehat A(-1+t)v_{\infty,\ell}(t)$,
so that $c_{\infty,\ell}(t)=t^{\kappa_{\infty,\ell}}\widetilde c_{\infty,\ell}(t)$,
and set
\begin{equation}
\gamma_{\infty,\ell}(\widehat A)=
[t^{\kappa_{\infty,\ell}}]c_{\infty,\ell}(t),
\quad
\Gamma_\infty(\widehat A)=
[\gamma_{\infty,1}\ \cdots\ \gamma_{\infty,2n}].
\label{eq:gamma-infty}
\end{equation}
Thus $\Gamma_\infty(\widehat A)$ is $n\times 2n$.  Let
\begin{equation*}
E_D=\begin{bmatrix}0_{n\times n}\\I_n\end{bmatrix}
\in\mathbb R^{2n\times n}
\end{equation*}
select the last $n$ columns.  Evaluating \eqref{eq:smith-product} at $t=0$
gives
$[\,\widehat N_A(-1)\;\widehat D_A(-1)\,]
=\Gamma_\infty(\widehat A)L_\infty(0)$, and therefore
\begin{equation}
D_A(\infty)=\Gamma_\infty(\widehat A)L_\infty(0)E_D.
\label{eq:D-infty}
\end{equation}
Controller properness is the finite inequality $\det D_A(\infty)\neq0$.
The identity remains valid when some or all of the Smith indices at
infinity vanish: if $\kappa_{\infty,\ell}=0$, the coefficient
$\gamma_{\infty,\ell}$ is simply the constant term of $c_{\infty,\ell}$.
Even if $-1$ is not a zero of $\det\widehat{\mathcal M}$ at all, the point
must still be retained in \eqref{eq:nodes} in order to evaluate the
denominator.  This extra order at $z=-1$ is not optional.  Ordinary nodes
need only the vanishing conditions \eqref{eq:smith-hermite}; infinity
requires one further jet coefficient.

The reconstruction row produced by the standing normalization is left
coprime, so properness of $K=D_A^{-1}N_A$ is equivalent to nonsingularity of
$D_A(\infty)$.  Thus \eqref{eq:D-infty} is the complete finite expression of
controller properness for the segment.

\subsection{Normalized jets, lengths, and radial return}

For an analytic matrix function $F$, the \emph{normalized jet of order} $m$
at $\alpha$ is the truncated Taylor list
\begin{equation*}
\mathcal J_\alpha^m F=(F_0,\ldots,F_m),
\qquad
F_q=\frac{F^{(q)}(\alpha)}{q!}.
\end{equation*}
Thus $\mathcal J_\alpha^m F$ records the coefficients of $t^0$ through
$t^m$.  Ordinary nodes require the vanishing conditions
\eqref{eq:smith-hermite}, hence the lengths
\begin{equation*}
m_j=\max_{1\leq\ell\leq 2n}\kappa_{j\ell}
\end{equation*}
away from $-1$: one retains the jets $\mathcal J_{\zeta_j}^{m_j-1}$.  At
infinity the extra coefficient in \eqref{eq:gamma-infty} forces
\begin{equation*}
m_\infty=1+\max_{1\leq\ell\leq 2n}\kappa_{\infty,\ell},
\end{equation*}
so one retains $\mathcal J_{-1}^{m_\infty-1}$.  These lengths are read from
the endpoint block $\widehat{\mathcal M}$.  They do not depend on the
McMillan degree of an unknown controller.  The chains $\nu_{j\ell}$,
$v_{j\ell}$ and the germs $L_j$ are input data; only finitely many of their
Taylor coefficients enter.  Products, inverses, and Cayley transforms of
such jets are finite convolutions in a truncated power-series ring;
Appendix~\ref{app:smith} records the algebra.

The nodes $\zeta_j$ may lie on the unit circle, whereas the confluent Pick
matrices used below require strictly interior nodes.  After
Theorem~\ref{thm:double-cayley}, the factors $U$, $Q_1$, and $Q_2$ are
analytic on a neighborhood of the closed disk, so a radius $r>1$
sufficiently close to one keeps $w\mapsto U(rw)$ unimodular in
$\mathcal R^{n\times n}$ and $w\mapsto Q_k(rw)$ strictly Schur, while
pushing every node to
\begin{equation*}
\alpha_j=\zeta_j/r\in\mathbb D.
\end{equation*}
If a matrix function $F$ has the expansion
$F(\alpha_j+\eta)=\sum_{q\geq0}F_{j,q}\eta^q$, radial return to the original
node produces
\begin{equation}
F\bigl((\zeta_j+t)/r\bigr)
=\sum_{q\geq0}r^{-q}F_{j,q}t^q.
\label{eq:radial-return}
\end{equation}
The Smith--Hermite equations \eqref{eq:smith-hermite} and the infinity
identity \eqref{eq:D-infty} are therefore imposed on the scaled coefficients
$r^{-q}F_{j,q}$, not on the interior jets alone.  Substituting the interior
nodes $\alpha_j$ for $\zeta_j$ while leaving Taylor coefficients unscaled
yields an incorrect criterion.

\subsection{Two lift doorways}

The finite data at the scaled nodes consist of conjugate-symmetric
$n\times n$ jets for a prospective unimodular factor $U$ and for two
prospective Schur factors $W_1,W_2$.  Two classical lifts convert those
jets into global rational matrices.  Neither lift is a simultaneous
stabilization theorem; each is recorded here only as a doorway from finite
jets to functions on the disk.

\begin{prop}[Unimodular Hermite lift]\label{prop:hermite-lift}
Let $\Lambda\subset\overline{\mathbb D}$ be a finite conjugation-invariant
node set.  At each node prescribe a conjugate-symmetric $n\times n$ jet
\begin{equation*}
\mathbf U_\alpha(\eta)=\sum_{q=0}^{m_\alpha-1}U_{\alpha,q}\eta^q,
\qquad \det U_{\alpha,0}\neq0.
\end{equation*}
There exists $U\in\mathrm{GL}_n(\mathcal R)$ matching all jets if and only
if the determinants $\det U_{\alpha,0}$ at all real nodes have the same
nonzero sign.  If there is no real node, the sign condition is vacuous.
When $n=1$ the construction is scalar.  For $n\geq2$ the lift uses the
classical identity $\mathrm{SL}_n(B)=E_n(B)$ on the associated Artin ring
\cite{Bass1968algebraic}.
\end{prop}

Necessity is elementary: if such a $U$ exists, then $\det U$ is continuous,
real-valued, and nonzero on the compact interval $[-1,1]$, hence of constant
sign.  For sufficiency, the identity $\mathrm{SL}_n(B)=E_n(B)$ on the Artin
ring $B=\mathbb R[z]/(p)$ associated with the node polynomial $p$ is
classical generation \cite{Bass1968algebraic}.  The if-and-only-if sign
criterion is the assembly that combines this generation with a scalar
determinant lift having no zero in the closed disk; when $n=1$ the
special-linear step is absent.  The proof is given in
Appendix~\ref{app:hermite}.

The Schur jets are interpolated by a standard confluent matrix
Nevanlinna--Pick theorem, not by an interpolation theory developed here.
If the associated confluent Pick matrix is positive definite, then there
exist real-rational Schur interpolants, and every real-rational Schur
interpolant matching the jets is given by a Redheffer linear-fractional
formula with a real-rational Schur parameter; see
Lemma~\ref{lem:supp-redheffer}.
The disk, Stein, and $J$-inner coefficient-matrix form used below is the
classical unit-disk theory of
\cite{BallGohbergRodman1990interpolation}; the bitangential
operator-argument theorem of \cite{Ball2018bitangential} is the
right half-plane counterpart and is not identified with the displayed
disk kernel.  Positive definiteness of the Pick matrix is the strict
feasibility condition.
Details of the specialization to the present jets---real coefficients under
conjugate symmetry, and rationality of the inverse parameter---are recorded
in Appendix~\ref{app:redheffer}.

Taken together, Proposition~\ref{prop:hermite-lift} and
Lemma~\ref{lem:supp-redheffer} produce global factors $U\in\mathrm{GL}_n(\mathcal R)$
and rational Schur $W_1,W_2$ that match the recorded interior jets exactly.
They do not, by themselves, produce an admissible matrix
$A\in\mathfrak A_{\mathrm{adm}}^{(n)}$.  A Pick-feasible pair of Schur jets
is not yet a feasible controller: the double-Cayley combination of the
lifted factors must still satisfy \eqref{eq:smith-hermite} and
\eqref{eq:D-infty} after the radial return \eqref{eq:radial-return}.  The
next section encodes those remaining conditions, together with the Hermite
sign condition and Pick positive definiteness, as one finite seed.

\section{Finite feasibility, exact decision, and synthesis}
\label{sec:decision}

Proposition~\ref{prop:reduction} reduced common stabilization of the
stated segment to nonemptiness of $\mathfrak A_{\mathrm{adm}}^{(n)}$.
Theorem~\ref{thm:double-cayley} factorized the remaining spectral-cut
constraint globally over $\mathcal R$.  Section~\ref{sec:smith} reduced
reconstruction stability and controller properness to finitely many
Smith--Hermite equations of input-determined jet length, together with
two classical lift doorways.  The two constraints now meet.  This
section compresses them, without loss for the existence question, into
one finite-dimensional semialgebraic seed.  Every admissible matrix
$A$, of whatever McMillan degree, produces an admissible seed, and
every admissible seed produces some admissible $A$.  Quantifier
elimination is applied only after that compression: it is the decision
engine, not the contribution.  A positive sample synthesizes one proper
real-rational common stabilizer.  Covering every remaining lift, rather
than constructing a single controller, is deferred to
Section~\ref{sec:atlas}.

\subsection{Admissible finite seeds}

The node set $\mathcal Z$, the local Smith chains
\eqref{eq:smith}--\eqref{eq:root-chain}, the cancellation equations
\eqref{eq:smith-hermite}, the infinity identity \eqref{eq:D-infty}, the
lengths $m_j$ and $m_\infty$, and the radial return
\eqref{eq:radial-return} are those of Section~\ref{sec:smith}.  In
particular $m_j=\max_\ell\kappa_{j\ell}$ away from $-1$ and
$m_\infty=1+\max_\ell\kappa_{\infty,\ell}$ at the node representing
infinity.  These integers are read from $\widehat{\mathcal M}$; they do
not depend on an unknown controller.  For $r>1$ the interpolation nodes
are the interior points $\alpha_j=\zeta_j/r\in\mathbb D$, including the
image of $-1$.  Normalized jets are those of
Section~\ref{sec:smith}: $\mathcal J_\alpha^m F=(F_0,\ldots,F_m)$ records
the coefficients of $t^0$ through $t^m$.

A \emph{finite seed} $\xi$ consists of:
\begin{itemize}
\item the jets $\mathcal J_{\alpha_j}^{m_j-1}U$ of a prospective
unimodular factor $U\in\mathcal R^{n\times n}$, hence $m_j$ coefficients
at each scaled ordinary node and $m_\infty$ coefficients at the image of
$-1$;
\item the jets $\mathcal J_{\alpha_j}^{m_j-1}W_k$ of two prospective
Schur factors $W_1,W_2$, of the same lengths;
\item two real scalars $0<\gamma_1,\gamma_2<1$.
\end{itemize}
All jets are recorded at the interior nodes $\alpha_j$, not at the
original nodes $\zeta_j$.  The unknown quantities in $\xi$ are therefore
finitely many matrix coefficients together with the two margins
$\gamma_k$.  The radius $r>1$ is a separate real parameter, not a
coordinate of $\xi$.

From these interior jets one forms, in the truncated power-series ring, the
double-Cayley combination
\begin{equation}
B(w)=U(w)^{-1}\mathcal C(\gamma_1 W_1(w))
\mathcal C(\gamma_2 W_2(w))U(w).
\label{eq:B-interior}
\end{equation}
The operations required to evaluate \eqref{eq:B-interior} through any
finite order---jet multiplication, the inverse-jet recursion
\eqref{eq:supp-inv-jet}, and the Cayley transform \eqref{eq:car}---are
finite algebraic operations.  No global interpolant is constructed at
this stage.  Writing $B(\alpha_j+\eta)=\sum_q B_{j,q}\eta^q$ and
invoking \eqref{eq:radial-return}, the candidate disk matrix
$\widehat A(z)=B(z/r)$ has $q$th coefficient $r^{-q}B_{j,q}$ at the
original node $\zeta_j$.  The Smith--Hermite and infinity conditions
are imposed on those scaled coefficients.

The seed $\xi$ is \emph{admissible} at the radius $r$ when the following
four finite conditions hold.
\begin{enumerate}
\item The jets are conjugate-symmetric.  Every zeroth-order coefficient
of $U$ is invertible.  At all real nodes, the determinants
$\det U_{j,0}$ are real, nonzero, and of the same sign.  If there is no
real node, the sign condition is vacuous.
\item The two confluent matrix Pick matrices associated with the jets
of $W_1$ and of $W_2$ are positive definite.
\item Substituting the jets of $U$ and of $Q_k=\gamma_k W_k$ into
\eqref{eq:B-interior}, and multiplying the $q$th Taylor coefficient at
each interior node by $r^{-q}$ as in \eqref{eq:radial-return}, produces
jets of $\widehat A$ at the original nodes $\zeta_j$ that satisfy every
Smith--Hermite equation \eqref{eq:smith-hermite}.
\item The determinant obtained from the infinity identity
\eqref{eq:D-infty}, evaluated on those same scaled coefficients, is
nonzero.
\end{enumerate}
The second condition is a strict interpolation test on the jets of
$W_k$.  The matrices that actually enter the Cayley product are
$Q_k=\gamma_k W_k$.  The scalars $\gamma_k$ therefore separate a closed
Schur interpolation parameter from the strict Schur margin required by
Theorem~\ref{thm:double-cayley}.  Replacing $\zeta_j$ by $\alpha_j$
while leaving Taylor coefficients unscaled yields an incorrect
criterion, as already recorded after \eqref{eq:radial-return}.

Let $\Psi_n(r,\xi)$ denote the conjunction of (1)--(4), viewed as a
finite first-order condition in the real and imaginary parts of the
seed coordinates and in the real variables $r,\gamma_1,\gamma_2$.  The
admissible seed set at radius $r$ is
\begin{equation*}
\mathscr J_{\mathrm{adm}}^{(n)}(r)
=\{\xi:\ \Psi_n(r,\xi)\}.
\end{equation*}
Thus $\mathscr J_{\mathrm{adm}}^{(n)}(r)$ is nonempty if and only if
there exists a seed $\xi$ with $\Psi_n(r,\xi)$.

\subsection{Finite seed criterion}

\begin{thm}[Finite seed criterion]\label{thm:seed}
Under the standing hypotheses,
\begin{equation*}
\mathfrak A_{\mathrm{adm}}^{(n)}\neq\emptyset
\quad\Longleftrightarrow\quad
\exists r>1\ \exists\xi:\ \Psi_n(r,\xi).
\end{equation*}
Equivalently, $\mathfrak A_{\mathrm{adm}}^{(n)}$ is nonempty if and
only if $\mathscr J_{\mathrm{adm}}^{(n)}(r)$ is nonempty for some real
$r>1$.
\end{thm}

The displayed quantification is over a real radius.  A rational radius
may be selected in the necessity argument below, because the unimodular
and strict Schur bounds persist on an open interval of radii and
$\mathbb Q$ is dense in that interval.  The predicate $r\in\mathbb Q$
is not first-order over a real closed field and is not part of
$\Psi_n$.  Section~\ref{sec:atlas} uses rational radii for a different
purpose, namely enumeration of atlas centers inside a fixed coefficient
field.

\begin{proof}
\emph{Necessity.}
Let $A\in\mathfrak A_{\mathrm{adm}}^{(n)}$ be arbitrary, of any McMillan
degree.  Theorem~\ref{thm:double-cayley} supplies a double-Cayley
factorization whose factors are analytic on a neighborhood of the closed
disk.  A slight radial expansion $w=rz$ with $r>1$ keeps $U$ unimodular
and $Q_k$ strictly Schur; the jet lengths $m_j,m_\infty$ are read from
the input block $\widehat{\mathcal M}$ and do not grow with the degree of
$A$.  Choose $\lVert Q_k\rVert_\infty<\gamma_k<1$ and set
$W_k=Q_k/\gamma_k$, so each $W_k$ is strict Schur (hence its confluent
Pick matrix is positive definite) while
$\mathcal C(\gamma_k W_k)=\mathcal C(Q_k)$.  Recording the jets of
$U,W_1,W_2$ yields a seed $\xi$.  Radial return
\eqref{eq:radial-return} recovers the original Smith--Hermite and infinity
data, so $\Psi_n(r,\xi)$ holds.  This is a
lossless compression, not an inversion of the lift below.

\emph{Sufficiency.}
Seed condition~(1) is the hypothesis of
Proposition~\ref{prop:hermite-lift}; condition~(2) is $\Pi>0$ for each
$W_k$ in Lemma~\ref{lem:supp-redheffer}.  The seed is therefore lifted by
those two statements.  Substituting $Q_k=\gamma_k W_k$ into
\begin{equation}
\widehat A(z)=U(z/r)^{-1}\mathcal C(Q_1(z/r))
\mathcal C(Q_2(z/r))U(z/r)
\label{eq:A-from-seed}
\end{equation}
produces a matrix in
$\mathcal A_{\mathrm{cut}}^{(n)}$.  Because the lifts match the recorded
jets, the scaled coefficients satisfy \eqref{eq:smith-hermite} and
\eqref{eq:D-infty}.  Hence $A\in\mathfrak A_{\mathrm{adm}}^{(n)}$.
Positive definiteness of the Pick matrices of $W_k$ is not controller
feasibility: the Smith--Hermite tests are imposed on the double-Cayley
combination $U^{-1}\mathcal C(Q_1)\mathcal C(Q_2)U$, not on each interpolant
separately.

The two directions do not depend on $n=2$.  The only size changes are
that $\widehat{\mathcal M}$ has $2n$ Smith chains and each Pick matrix has
state dimension $d=n\sum_j m_j$.
\end{proof}

\subsection{Unknowns, matrix sizes, and the first-order formula}

The unknown quantities in $\Psi_n$ are the finitely many jet coefficients
of $U,W_1,W_2$ together with the real scalars $r,\gamma_1,\gamma_2$.
Nodes, Smith indices, and jet lengths are input.  As already recorded after
the proof of Theorem~\ref{thm:seed}, the only $n$-dependent sizes are that
$\widehat{\mathcal M}$ is $2n\times 2n$ and each Pick matrix has state
dimension
\begin{equation*}
d=n\sum_j m_j.
\end{equation*}
The sentence $\Psi_n$ is a finite first-order formula over the reals
\cite{BasuPollackRoy2006algorithms}, with $r>1$ a real variable; it does
not contain the non-semialgebraic clause $r\in\mathbb Q$.  The Stein
solution $\Pi_k$ may be introduced as an existentially quantified
positive-definite unknown satisfying
$\Pi_k-T(r)\Pi_k T(r)^*=XX^*-Y_kY_k^*$ with nodes $\alpha_j=\zeta_j/r$,
or equivalently as the unique rational function of $(r,\xi)$ obtained by
inverting $I-T\otimes\overline T$.  Inverse jets are encoded by the
recursion \eqref{eq:supp-inv-jet}.  Nonreal determinants are tested by
$(\operatorname{Re}\delta)^2+(\operatorname{Im}\delta)^2>0$,
not by $\delta^2>0$.  The finitely many Taylor coefficients of the local
Smith germs that enter the formula lie in a finite real-algebraic
extension of the plant field; $\Psi_n$ is formed over that still
effectively presented extension.

\subsection{Exact decision and synthesis}

An algorithmic statement requires an exact coefficient model.  Let
$\mathbb F\subset\mathbb R$ be an effectively presented real closed
field: field operations, equality and order tests, root isolation,
finite algebraic extensions, and real quantifier elimination are all
effective.  Real algebraic numbers are the principal example.  Assume
that the coefficients of \eqref{eq:plant-line} lie in $\mathbb F$ and
that stable denominators of the endpoint factors are presented as
certified input, together with the standing hypotheses.

\begin{cor}[Exact decision]\label{cor:algorithm}
Under those hypotheses, there is a finite terminating exact procedure
that decides whether $\mathfrak A_{\mathrm{adm}}^{(n)}$ is nonempty.
A positive answer constructs one proper real-rational common stabilizer
of \eqref{eq:plant-line}.  A negative answer is obtained by deciding
exactly that the finite first-order seed formula $\Psi_n$ is false.  No
bound on the controller McMillan degree is an input to the procedure.
A proof-producing quantifier-elimination implementation may retain its
projection and sign data as an implementation certificate; no
implementation-independent certificate format is asserted.
\end{cor}

The existence decision does not presuppose or search over controller
McMillan degree.  Some auxiliary constructions after a positive sample
still use a terminating enumeration.  Those two facts are not the same,
and the second does not restore a degree bound as an input of the
decision.

\begin{proof}
Form $\Psi_n$ from the local Smith data of $\widehat{\mathcal M}$ over
$\mathbb F$, with a real variable $r>1$.  Quantifier elimination
\cite{Tarski1951decision,Collins1975QE} decides the sentence.  If it is
false, Theorem~\ref{thm:seed} yields that no proper real-rational common
stabilizer exists.  If it is true, Lemma~\ref{lem:supp-hermite-effective}
lifts the $U$ jets (enumerating an auxiliary scalar polynomial $q$, not a
controller degree) and the central Redheffer interpolants $R=0$ lift
$W_1,W_2$.  Equation \eqref{eq:A-from-seed} and
Proposition~\ref{prop:reduction} return one controller.  The procedure is
confined to the present structured class; it does not decide simultaneous
stabilization of three general plants
\cite{BlondelGevers1993undecidable}.
\end{proof}

The following procedure records the same construction.  It promises
only the termination already proved.  It does not claim a running-time
bound, a minimal McMillan degree, or an implementation-independent
infeasibility certificate.

\begin{description}
\item[Input.]
Endpoint factors $N_i,D_i$ with coefficients in $\mathbb F$, certified
stable denominators, and the standing hypotheses.
\item[Output.]
Either a proper real-rational common stabilizer $K$ of
\eqref{eq:plant-line}, or the decision that
$\mathfrak A_{\mathrm{adm}}^{(n)}=\emptyset$.
\item[Procedure.]
\begin{enumerate}
\item Clear certified stable denominators.  Isolate the zeros of
$\det\widehat{\mathcal M}$ in $\overline{\mathbb D}$ and compute the
local Smith data and the jet lengths $m_j$, $m_\infty$.
\item Form the finite first-order sentence
$\exists r>1\,\exists\xi\,\Psi_n(r,\xi)$ over $\mathbb F$, with $r$ a
real variable.  Do not restrict $r$ to $\mathbb Q$.
\item Decide the sentence by real quantifier elimination
\cite{Tarski1951decision,Collins1975QE}.
\item If the sentence is false, return that
$\mathfrak A_{\mathrm{adm}}^{(n)}$ is empty.
\item If it is true, extract a feasible sample $(r,\xi)$.  Construct
$U\in\mathrm{GL}_n(\mathcal R)$ by the effective Hermite procedure of
Lemma~\ref{lem:supp-hermite-effective}, which enumerates an auxiliary
scalar polynomial $q$ until $d=H+pq$ is zero-free on the closed disk.
\item Construct $W_1,W_2$ by the central Redheffer interpolants $R=0$,
form $Q_k=\gamma_k W_k$, and set $\widehat A$ by
\eqref{eq:A-from-seed}.
\item Return $K=D_A^{-1}N_A$.
\end{enumerate}
\end{description}

Quantifier elimination over real closed fields may be costly, and the
degree of the auxiliary polynomial $q$ accepted by the Hermite
enumeration may be large.  The McMillan degree of the controller
produced in the positive case is therefore not claimed to be minimal,
and no uniform complexity bound is asserted.  If a particular
elimination implementation returns projection polynomials and sign
conditions, those data may be retained as evidence of that
implementation; announcing that a finite formula is false is not the
same as exhibiting a short certificate that is independent of the
implementation and convenient for hand checking.

The statement is confined to the present nondegenerate square affine
right-coprime factor segment.  It does not assert that simultaneous
stabilization of three general plants is decidable
\cite{BlondelGevers1993undecidable}.  The finite seed remains a
constrained semialgebraic set: Theorem~\ref{thm:seed} decides emptiness
or nonemptiness of $\mathfrak A_{\mathrm{adm}}^{(n)}$, and the
procedure above returns one controller in the nonempty case.  The
remaining unimodular and Schur lifting freedom, and the covering of
every proper real-rational common stabilizer, are the subject of
Section~\ref{sec:atlas}.

\section{A surjective atlas of all common stabilizers}
\label{sec:atlas}

Section~\ref{sec:decision} answered an existence question: a finite
semialgebraic seed decides whether a proper common stabilizer exists and,
when the answer is positive, constructs one.  The present section retains,
for each feasible seed, two classical Schur lifts and a residual unimodular
degree of freedom, and ranges over rational radii and admissible seeds, so
as to cover every proper real-rational common stabilizer of the stated
$n\times n$ affine right-coprime factor segment.

A coefficient-domain distinction must be recorded before the atlas is
stated, rather than patched afterwards.  In the finite decision formula of
Theorem~\ref{thm:seed} and Corollary~\ref{cor:algorithm}, the radius enters
as a real variable: one quantifies $\exists r>1$ in a first-order
semialgebraic condition.  The restriction $r\in\mathbb Q$ is not a
semialgebraic predicate and must not be inserted into that seed formula.
For the atlas the same strict Schur and unimodular margins that justify
radial expansion also permit a lossless passage to rational radii.  Every
admissible double-Cayley factorization is analytic on a neighborhood of the
closed disk, so the admissible radii contain an open interval
$(1,1+\varepsilon)$.  Intersecting with $\mathbb Q$ therefore omits no
controller.  With $r\in\mathbb Q\cap(1,\infty)$, the interpolation nodes
$\alpha_j=\zeta_j/r$ and the associated real Blaschke data remain in the
fixed effectively presented field $\mathbb F$ of
Section~\ref{sec:decision}.  The countable unimodular centers below can
then be enumerated over $\mathbb F(z)$.  Admissible seeds and the three
Schur parameters may still have arbitrary real coefficients; only the
centers are restricted to $\mathbb F(z)$.  Density of those centers is what
yields coverage of every real-rational common stabilizer.

Fix such a rational radius $r>1$ and an admissible seed
$\xi\in\mathscr J_{\mathrm{adm}}^{(n)}(r)$.  The two strictly feasible
confluent Pick problems determined by the $W_1$ and $W_2$ jets of $\xi$
have rational $J$-inner coefficient matrices, where
$J=\operatorname{diag}(I_n,-I_n)$,
\begin{equation*}
\Theta_{k,\xi}=
\begin{bmatrix}
\Theta_{11}^{(k)}&\Theta_{12}^{(k)}\\
\Theta_{21}^{(k)}&\Theta_{22}^{(k)}
\end{bmatrix},\qquad k=1,2,
\end{equation*}
obtained from one finite Stein equation in the disk realization of
\cite{BallGohbergRodman1990interpolation}.
Every real-rational Schur lift is given by the classical Redheffer formula
\cite{BallGohbergRodman1990interpolation}
\begin{equation}
W_k(R_k)=(\Theta_{11}^{(k)}R_k+\Theta_{12}^{(k)})
\bigl(\Theta_{21}^{(k)}R_k+\Theta_{22}^{(k)}\bigr)^{-1},
\quad \lVert R_k\rVert_\infty\leq1.
\label{eq:redheffer}
\end{equation}
Each block of $\Theta_{k,\xi}$ is $n\times n$, so $\Theta_{k,\xi}$ is
$2n\times 2n$.  The Redheffer denominator is automatically unimodular over
$\mathcal R$.  Changing $R_k$ in the closed real-rational Schur class
preserves both the prescribed jets and the Schur bound.  We use
\eqref{eq:redheffer} as a standard interpolation tool; the contribution
below is its placement inside the seed atlas, not the interpolation theorem
itself.  The disk, confluent-jet, and real-rational specialization used here
is recorded in Appendix~\ref{app:redheffer}.

The unimodular factor is different.  Proposition~\ref{prop:hermite-lift}
supplies one unimodular Hermite lift $U_\xi\in\mathrm{GL}_n(\mathcal R)$ of
the $U$ jets of $\xi$.  Any other matching lift may be written uniquely as
$U=U_\xi E$, where $E$ has identity jets
\begin{equation*}
\mathcal J_{\alpha_j}^{m_j-1}E=\mathcal J_{\alpha_j}^{m_j-1}I_n
\quad\text{for all }j.
\end{equation*}
Thus the residual freedom is exactly the identity-jet subgroup of
$\mathrm{GL}_n(\mathcal R)$.  Let
$b(z)=\prod_j(z-\alpha_j)^{m_j}\in\mathbb R[z]$, and let $\beta$ be a finite
real Blaschke product with the same zeros and multiplicities,
\begin{equation*}
\beta(z)=\eta\prod_j
\left(\frac{z-\alpha_j}{1-\bar\alpha_jz}\right)^{m_j},
\qquad \lvert\eta\rvert=1,
\end{equation*}
the unimodular constant $\eta$ being chosen so that the coefficients of
$\beta$ are real.  Then $\lvert\beta\rvert=1$ on the unit circle.  Identity
jets are equivalent to $E-I_n=bG$ for some $G\in\mathcal R^{n\times n}$.
Enumerating
$G_\nu\in\mathbb F(z)^{n\times n}\cap\mathcal R^{n\times n}$ and retaining
only those candidates for which
$E_\nu=I_n+bG_\nu\in\mathrm{GL}_n(\mathcal R)$ produces a countable family
of unimodular centers with identity jets.  After stable denominators are
cleared, membership in this family is a first-order test over $\mathbb F$.
There is then a surjective covering
\begin{equation}
\begin{split}
&\{E\in \mathrm{GL}_n(\mathcal R):
  \mathcal J_{\alpha_j}^{m_j-1}E=\mathcal J_{\alpha_j}^{m_j-1}I_n\}\\
&\quad=\bigcup_{\nu\geq0}
\{E_\nu\mathcal C(\beta R_0):
  \lVert R_0\rVert_\infty<1\}.
\end{split}
\label{eq:E-atlas}
\end{equation}
A small-norm neighborhood of the identity is used only to select a center
in the covering argument; every strict real-rational Schur parameter
$R_0$ is admissible on the right-hand side of \eqref{eq:E-atlas}.
Density of the centers and the local Cayley--Blaschke chart are recorded
in Appendix~\ref{app:atlas}.

The total reconstruction is as follows.  Let $U_\xi$ be one unimodular
Hermite lift of $\xi$, and set
\begin{equation}
\begin{split}
U&=U_\xi E_\nu\mathcal C(\beta R_0),\qquad
Q_k=\gamma_k W_k(R_k),\\
\widehat A(z)&=U(z/r)^{-1}\mathcal C(Q_1(z/r))
\mathcal C(Q_2(z/r))U(z/r).
\end{split}
\label{eq:A-atlas}
\end{equation}
Transporting $\widehat A$ back to the half-plane by
$s=(1-z)/(1+z)$ and applying \eqref{eq:KfromA} produces a controller $K$.

\begin{thm}[All-controller atlas]\label{thm:atlas}
Let $r$ range over $\mathbb Q\cap(1,\infty)$ and let
$\xi$ range over $\mathscr J_{\mathrm{adm}}^{(n)}(r)$.  For each such pair,
let $\nu\geq0$ and let $R_0,R_1,R_2$ range over the real-rational Schur
classes specified in \eqref{eq:redheffer}--\eqref{eq:E-atlas}.  Equations
\eqref{eq:A-atlas} and \eqref{eq:KfromA} produce only proper common
stabilizers of the stated $n\times n$ segment.  As $(r,\xi,\nu,R_0,R_1,R_2)$
range as above, every proper real-rational common stabilizer of that
segment arises at least once.  The resulting map is surjective and need
not be injective.
\end{thm}

\begin{proof}
(i) Every output is legal.  The parameters $R_1,R_2$ preserve the $W_k$
jets of $\xi$ by the classical Redheffer parameterization
\eqref{eq:redheffer} \cite{BallGohbergRodman1990interpolation}.  The identity-jet covering \eqref{eq:E-atlas} preserves
the $U$ jets, so $U$ remains a unimodular Hermite lift of those jets.
Because $0<\gamma_k<1$, each Cayley input $Q_k$ is strictly Schur even when
$R_k$ merely satisfies $\lVert R_k\rVert_\infty\leq1$.  Invoking the reverse
inclusion in Theorem~\ref{thm:double-cayley} therefore yields spectral-cut
avoidance on the closed disk.  The scaled jets of $\widehat A$ coincide with
those encoded by $\xi$, so the Smith--Hermite equations and the infinity
determinant condition of $\xi$ are inherited.  Proposition~\ref{prop:reduction}
then returns a proper real-rational common stabilizer of the stated
segment.

(ii) Every proper real-rational common stabilizer is recovered.  Normalize
an arbitrary such controller by Proposition~\ref{prop:reduction}, obtaining
its unique matrix $A\in\mathfrak A_{\mathrm{adm}}^{(n)}$.  Invoke
Theorem~\ref{thm:double-cayley} in the disk coordinate.  The resulting
unimodular and strict Schur factors are analytic on a neighborhood of the
closed disk, so a rational radius $r>1$ sufficiently
close to one may be chosen, the factors may be radially rescaled, and a
finite admissible seed $\xi$ may be recorded; this is the necessity
direction of Theorem~\ref{thm:seed}, with the lossless rational restriction
on $r$ already justified above.  Invoking the inverse Redheffer
transformation recovers real-rational parameters $R_1,R_2$ for the two
Schur lifts.  The residual $U_\xi^{-1}U$ has identity jets, so
\eqref{eq:E-atlas} recovers an index $\nu$ and a strict Schur parameter
$R_0$.  Substitution into \eqref{eq:A-atlas} reproduces the normalized $A$,
and \eqref{eq:KfromA} therefore reproduces the original controller.
Density of the enumerated centers and the local Cayley chart used to
construct \eqref{eq:E-atlas} are proved in Appendix~\ref{app:atlas}.  The
increment is this surjective atlas for the stated $n\times n$ segment; the
map need not be injective.
\end{proof}

\begin{rem}
The finite seed remains constrained by the semialgebraic conditions of
Theorem~\ref{thm:seed}.  After a feasible seed and a center are fixed, the
three Schur-class functions may vary freely within their stated classes,
but they do not furnish a single unconstrained and unique Youla coordinate
for the original problem
\cite{Obinata1988characterization}.  Uniqueness of the normalized matrix
$A$ associated with a given controller, as in
Proposition~\ref{prop:reduction}, does not make the intermediate atlas
coordinates unique: the same controller may arise from several
double-Cayley factorizations, several admissible radii, several seeds, and
several centers.  The index $\nu$ merely enumerates the countable family of
unimodular centers; it is not an index of connected components of the
identity-jet subgroup.  The branch centers and the three Schur parameters
may have arbitrarily large McMillan degree.
\end{rem}

\section{Exact examples and separation results}
\label{sec:examples}

This section records three exact instances, organized by the claim each one
supports.  A displayed controller together with a closed-loop identity proves
feasibility of that controller on that segment.  A prescribed finite seed,
with its Smith jets and Pick matrices, proves that the seed is admissible.
None of the instances below is presented as an automatic
quantifier-elimination solve of the decision formula in
Theorem~\ref{thm:seed}.  Discrete frequency and parameter samples are
visualization only: they do not replace an identity on the closed disk and
the continuum of $\lambda$.

\subsection{\texorpdfstring{A repeated-zero $2\times 2$ feasible instance ($E+$)}{A repeated-zero 2 x 2 feasible instance (E+)}}

This instance certifies that a repeated Smith chain, a nonzero infinity
index, and an explicit finite seed can occur together on a feasible
$2\times 2$ segment.  In the disk coordinate, define
\begin{equation*}
\begin{split}
p(z)&=\frac{z^2(1+z)}8,\qquad
q(z)=-\frac{1+3z}8,\\
h(z)&=\frac{(1+z)^2}{64},\qquad
G(z)=\begin{bmatrix}p(z)&h(z)\\0&q(z)\end{bmatrix},\\
C(z)&=-4(1+z)I_2.
\end{split}
\end{equation*}
Take
\begin{equation*}
(N_0,D_0)=(0,I_2),\qquad
(N_1,D_1)=(G,I_2-CG).
\end{equation*}
For every $\lambda\in[0,1]$,
$CN_\lambda+D_\lambda=I_2$, so the affine factor pair is right coprime.
At $z=-1$, $C(-1)=0$ and hence $D_\lambda(-1)=I_2$, so every plant is proper.
The nonconstant controller row
\begin{equation*}
[\,N_c\;D_c\,]=[\,I_2+C\;I_2\,]
\end{equation*}
gives
\begin{equation}
K(s)=\frac{s-7}{s+1}I_2,\qquad
\Delta_\lambda=I_2+\lambda G.
\label{eq:positive-controller}
\end{equation}
This is already an exact direct certificate: $\Delta_\lambda$ is upper
triangular and, on $\lvert z\rvert\leq1$,
$\lvert p(z)\rvert\leq1/4$ and $\lvert q(z)\rvert\leq1/2$.  Both diagonal
entries are therefore bounded away from zero for the entire continuum of
$\lambda$.  All entries are polynomial, and the inverse remains analytic on a
neighborhood of the closed disk.

The instance exercises repeated zeros and infinity.  In fact,
\begin{equation*}
\det\mathcal M=\det G
=-\frac{z^2(1+z)(1+3z)}{64},
\end{equation*}
and the full local Smith exponent tuples of $\widehat{\mathcal M}$ are
$(0,0,0,2)$ at $z=0$, $(0,0,0,1)$ at $z=-1/3$, and $(0,0,0,1)$ at
$z=-1$, the last point representing infinity.  The polynomials $p,h,q$
are linearly independent, since their respective degrees are $3,2,1$.
Consequently the coefficient matrices of $G$ span the three-dimensional
space generated by $E_{11},E_{12},E_{22}$.  If constant nonsingular matrices
$L,R$ made $LP_\lambda(z)R$ diagonal for every $\lambda$, differentiating
$P_\lambda=\lambda G(I_2-\lambda CG)^{-1}$ at $\lambda=0$ would make
$LG(z)R$ diagonal for every $z$.  Left and right multiplication by
nonsingular constant matrices preserve the dimension of the coefficient
span, whereas diagonal $2\times2$ matrices span a space of dimension two.
This contradiction rules out such a fixed decoupling.  Moreover, the
$(2,2)$ entry of $D_1$ has numerator
$1-4z-3z^2$ and a zero $(-2+\sqrt7)/3\in(0,1)$; the endpoint plant is
therefore not merely a stable plant disguised by the factors.

For an exact finite-seed certificate, let $r=6/5$,
$G_r(w)=G(rw)$, and choose
\begin{equation}
\begin{split}
&U=I_2,\qquad Q_1=W_1=0,
\qquad \gamma_1=\gamma_2=\frac12,\\
&Q_2=G_r(2I_2+G_r)^{-1},\qquad W_2=2Q_2.
\end{split}
\label{eq:positive-seed}
\end{equation}
Then
$\mathcal C(Q_2)=I_2+G_r$, and evaluation at $w=z/r$ recovers
$A=I_2+G$.  The interpolation nodes are
\begin{equation*}
0,\quad -\frac5{18},\quad -\frac56,
\qquad\text{with multiplicities }2,1,2,
\end{equation*}
which are the retained jet lengths at $0$, $-1/3$, and infinity after the
radial scaling.  Exact entrywise bounds give
\begin{equation}
\lVert W_2\rVert_\infty
\leq2\sqrt{\frac{476779337}{2089769796}}
<1.
\label{eq:positive-schur-bound}
\end{equation}
The two confluent Pick matrices are $10\times10$; all their leading principal
minors are positive.  Feasibility is proved by
\eqref{eq:positive-controller} and by this seed, not by frequency sampling.

\begin{figure}[t]
\centering
\includegraphics[width=0.97\textwidth]{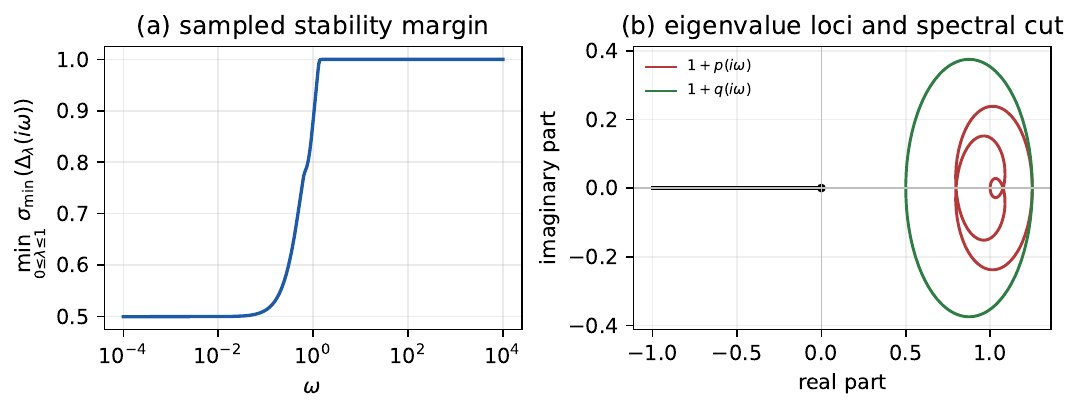}
\caption{Numerical illustration for the feasible instance $E+$.  (a)~The
sampled quantity $\min_{\lambda}\sigma_{\min}(\Delta_\lambda(i\omega))$ over a
$301$-point grid in $\lambda$; $500$ logarithmic frequency samples give a
minimum of approximately $0.499258$.  This is a sampled invertibility
quantity in the chosen factor representation, not an intrinsic robust
stability radius unless a further argument is supplied.  (b)~The two
eigenvalue loci of $I_2+G(z(i\omega))$ and the excluded spectral cut.  The
plot is a numerical illustration only.  Feasibility is proved by the
identity \eqref{eq:positive-controller} and by the exact seed
\eqref{eq:positive-seed}--\eqref{eq:positive-schur-bound}, not by the plot.}
\end{figure}

\subsection{Strict separation from a single-square-root class}

This instance certifies a strict inclusion between a comparison class defined
presently and the set of all proper common stabilizers of one and the same
segment.  It does not prove that the segment cannot be stabilized by a
square-root construction, and it does not refute an unidentified method from
the literature.

For a proper common stabilizer $K$ of a $2\times 2$ segment, let $A_K$ denote
the matrix obtained after the normalization $\Delta_0=I_2$ in
Proposition~\ref{prop:reduction}.  This matrix is unique: two stable
left-coprime representations of the same controller differ by a left
unimodular multiplier, and the normalization forces that multiplier to be
$I_2$.  Define
\begin{equation*}
\begin{split}
\mathfrak A_{\mathrm{sq}}&=\{A\in\mathfrak A_{\mathrm{adm}}^{(2)}:
 \exists B\in\mathcal H^{2\times2},\ A=B^2,\\
&\quad\widehat B(z)+\widehat B(z)^*>0\ (\lvert z\rvert\leq1)\},\\
\mathfrak K_{\mathrm{sq}}&=\{K:\ K\text{ is a proper common stabilizer}\\
&\quad\text{and }A_K\in\mathfrak A_{\mathrm{sq}}\},
\end{split}
\end{equation*}
and write $\mathfrak K_{\mathrm{all}}$ for the set of all proper
real-rational common stabilizers of the segment under discussion.  Thus
$\mathfrak K_{\mathrm{sq}}$ consists of those common stabilizers whose unique
normalized matrix is the square of a strictly accretive real-rational matrix.
The class is defined here; it is of the same type as a single-square-root
ansatz, but it is not identified with the constructions of
\cite{Cui2024multivariable,Cui2024general}.

Let
\begin{equation*}
\begin{aligned}
x(s)&=\frac{(s-1)(s-2)}{4(s+1)^2},\\
\alpha(s)&=1+x(s),& \eta(s)&=2+x(s),
\end{aligned}
\end{equation*}
\begin{equation*}
\begin{aligned}
V(s)&=\begin{bmatrix}1&\eta(s)\\0&1\end{bmatrix},\\
A_*(s)&=V(s)\operatorname{diag}(\alpha(s),2)V(s)^{-1},
\end{aligned}
\end{equation*}
and put $P=A_*-I_2$, $N_0=0$, $D_0=I_2$, $N_1=P$, and $D_1=I_2$.  Then
$\det\mathcal M=\det P=x$, whose relevant zeros are the simple points
$s=1,2$, and $\mathcal M(\infty)$ is nonsingular.  Explicitly,
\[
A_*=\begin{bmatrix}1+x&(2+x)(1-x)\\0&2\end{bmatrix},\qquad
P=\begin{bmatrix}x&(2+x)(1-x)\\0&1\end{bmatrix}.
\]

On the imaginary axis,
\begin{equation*}
\left\lvert\frac{(i\omega-1)(i\omega-2)}{(i\omega+1)^2}\right\rvert^2
=\frac{\omega^2+4}{\omega^2+1}\leq4.
\end{equation*}
The maximum modulus principle gives $\lvert x(s)\rvert\leq1/2$ on the closed
right half-plane.  Hence the controller $K_*=I_2$ has
$\Delta_\lambda=I_2+\lambda P$ with eigenvalues $1+\lambda x(s)$ and
$1+\lambda$, and it stabilizes the entire segment.

However, $A_*$ has no square root over $\mathbb R(s)$.  Any $B$ satisfying
$B^2=A_*$ commutes with $A_*$.  Since the eigenvalues $\alpha$ and $2$ are
distinct in $\mathbb R(s)$, the centralizer is diagonal in the $V$ basis, so
one diagonal entry would satisfy $b_1^2=\alpha$.  But
\begin{equation}
\alpha(s)=\frac{5s^2+5s+6}{4(s+1)^2},
\label{eq:alpha-not-square}
\end{equation}
the numerator has discriminant $-95$ and is irreducible in $\mathbb R[s]$, and
it is coprime to $s+1$.  As a rational function this irreducible factor has
odd valuation.  Every square in $\mathbb R(s)$ has even valuation at every
irreducible, so $\alpha$ is not a square and $K_*\notin\mathfrak K_{\mathrm{sq}}$.
Changing the left-coprime representation of $K_*$ cannot avoid the obstruction:
$A_{K_*}$ is unique.

The comparison class is nevertheless nonempty.  The second row of $P$ is
$[\,0\;1\,]$.  The constant controller
\begin{equation*}
K_0=\begin{bmatrix}0&2\\0&1\end{bmatrix}
\end{equation*}
therefore satisfies $K_0P=K_0$.  With the left-coprime row
$[\,N_c\;D_c\,]=[\,K_0\;I_2\,]$ one has
$\Delta_\lambda=I_2+\lambda K_0$, a constant invertible matrix, so $K_0$ is a
proper common stabilizer.  The normalization $\Delta_0=I_2$ is already in
force, and the unique normalized matrix is
\begin{equation*}
A_{K_0}=I_2+K_0P=I_2+K_0
=\begin{bmatrix}1&2\\0&2\end{bmatrix}
=:A_0.
\end{equation*}
Admissibility of $A_0$ does not require a separate Smith reconstruction:
$[\,K_0\;I_2\,]$ already lies in $\mathcal H^{2\times4}$, the spectrum
$\{1,2\}$ misses the cut $(-\infty,0]$, and $D_{A_0}=I_2$.  The constant
matrix
\begin{equation*}
B_0=\begin{bmatrix}1&2/(1+\sqrt2)\\0&\sqrt2\end{bmatrix}
\end{equation*}
satisfies $B_0^2=A_0$.  The leading principal minors of
$(B_0+B_0^T)/2$ are $1$ and $3(\sqrt2-1)>0$, so $B_0$ is strictly accretive.
Hence $A_0\in\mathfrak A_{\mathrm{sq}}$ and
$K_0\in\mathfrak K_{\mathrm{sq}}$.  In particular the same segment does admit
a square-root controller; what the example excludes is the membership of
$K_*$ in $\mathfrak K_{\mathrm{sq}}$.  One concludes
\begin{equation*}
\varnothing\neq\mathfrak K_{\mathrm{sq}}
\subsetneq\mathfrak K_{\mathrm{all}}.
\end{equation*}
Theorem~\ref{thm:double-cayley} permits two unequal strictly accretive
factors and therefore does not encounter the odd-valuation obstruction
\eqref{eq:alpha-not-square}.

\subsection{\texorpdfstring{An exactly infeasible instance ($E-$)}{An exactly infeasible instance (E-)}}

This instance certifies exact infeasibility of a nondegenerate affine
segment.  It is not used as a coupling obstruction.  Let
\begin{equation*}
\begin{split}
n(z)&=\frac{z(z-1/2)}{16},\qquad
d(z)=81-\frac{13120}{81}z,\\
S&=\frac15\begin{bmatrix}3&-4\\4&3\end{bmatrix},\qquad
H=\frac18\begin{bmatrix}0&1\\1&1\end{bmatrix},
\end{split}
\end{equation*}
and set
\begin{equation*}
(N_0,D_0)=(0,I_2),\qquad
(N_1,D_1)=(nS,dI_2+nH).
\end{equation*}
The standing assumptions hold.  For $\lambda>0$, $N_\lambda(z)$ is invertible
unless $n(z)=0$; at the two zeros $0$ and $1/2$, $D_\lambda$ is a positive
scalar matrix.  At infinity, $D_1(-1)$ is positive definite, and convexity
treats every $\lambda$.  Also
\begin{equation*}
\det\mathcal M=n(z)^2\not\equiv0.
\end{equation*}

Suppose a common proper controller existed, and normalize $\Delta_0=I_2$.
Since $N_0=0$ and $D_0=I_2$, the normalized denominator factor is $D_c=I_2$.
Therefore $A=\Delta_1=N_c N_1+D_1$ is forced to satisfy
\begin{equation}
A(0)=81I_2,\qquad A(1/2)=\frac1{81}I_2.
\label{eq:negative-forced-values}
\end{equation}
Spectral-cut avoidance makes the principal square root $G=A^{1/2}$ analytic
on the disk \cite{Higham2008functions}; rationality of $G$ is not required.
The matrix
\begin{equation*}
F=(G-I_2)(G+I_2)^{-1}
\end{equation*}
has all eigenvalues in the open unit disk.  The open disk is convex, so the
average of those eigenvalues lies in the open disk.  Hence
$\varphi=\tfrac12\operatorname{tr}F$ is a scalar Schur function.  Equation
\eqref{eq:negative-forced-values} imposes
\[
\varphi(0)=\frac45,\qquad \varphi(1/2)=-\frac45.
\]
With $\rho(a,b)=\lvert a-b\rvert/\lvert1-\bar b a\rvert$ denoting the
pseudohyperbolic distance, this contradicts the Schwarz--Pick inequality
because
\begin{equation*}
\rho(0,1/2)=\frac12
<\frac{40}{41}=\rho(4/5,-4/5).
\end{equation*}
Equivalently, the scalar Pick matrix is
\begin{equation*}
\begin{bmatrix}9/25&41/25\\41/25&12/25\end{bmatrix},
\qquad \det=-\frac{1573}{625}<0.
\end{equation*}
Thus the whole affine segment is exactly infeasible, not merely missed by a
chosen controller ansatz.

The infeasibility certificate does not use a coupling obstruction.  The
matrix $H$ is real symmetric, so a real orthogonal $O$ diagonalizes it:
$O^THO=\Lambda$.  The constant bases $L=O^TS^{-1}$ and $R=O$ give
\begin{equation*}
LP_\lambda R
=\lambda n(z)\bigl(((1-\lambda)+\lambda d(z))I_2+\lambda n(z)\Lambda\bigr)^{-1},
\end{equation*}
which is diagonal for every $\lambda$.  A fixed change of input--output bases
therefore splits the segment into two scalar channels.  The Schwarz--Pick
argument is independent of that splitting.

The three instances certify, respectively, a repeated-zero feasible seed, a strict inclusion of the single-square-root comparison class, and an exact infeasibility obstruction.  They do not constitute a complexity study.

\section{Scope and conclusion}
For every fixed finite $n\geq1$, simultaneous internal stabilization of a
nondegenerate square affine right-coprime factor segment reduces, by an
inherited coprime-factor normalization, to emptiness or nonemptiness of
$\mathfrak A_{\mathrm{adm}}^{(n)}$.  A function-level double-Cayley identity
then unlocks the spectral-cut constraint on that set: every admissible
matrix is, up to unimodular similarity over the stable rational ring, a
product of two Cayley transforms of strictly Schur functions.  Complete
Smith--Hermite jets, together with one additional Taylor coefficient at infinity,
are the finite reconstruction interface between this factorization and a
proper real-rational controller.  The two remaining constraints compress
losslessly into one finite-dimensional semialgebraic seed.  Over the
effective coefficient field of Corollary~\ref{cor:algorithm}, that seed
yields an exact existence decision that does not presuppose or search over
a controller McMillan-degree bound.  Ranging over admissible seeds of the
same encoding yields a surjective atlas of all proper real-rational common
stabilizers.

The theorems do not settle general MIMO simultaneous stabilization, do not
provide a unique unconstrained Youla parameter, and do not imply
semialgebraicity for arbitrary plant families.  Here $n$ is the
input--output channel size rather than the system order, and affinity of
the factors is not affinity of the plants.  Quantifier elimination may be
costly, and the examples of Section~\ref{sec:examples} are small exact
certificates rather than a complexity study.  These restrictions are
essential: the exact decision result is compatible with the rational
undecidability of simultaneous stabilization of three general linear plants
\cite{BlondelGevers1993undecidable}.

What remains outside the present theorems includes identically singular
endpoint blocks, multiparameter families, rectangular plants, degree
control, and practical solvers.

The priority claim is therefore limited to the stated function-level
factorization and its lossless combination with the stated nondegenerate
square affine coprime-factor segment for every finite $n$.

\appendix
\input{appendix-proofs}

\bibliographystyle{plain}
\bibliography{references}

\end{document}

%% file: appendix-proofs.tex
\section{Proof of Theorem~\ref{thm:double-cayley}}
\label{app:double-cayley}

The reverse inclusion is the argument in Section~\ref{sec:cayley}: the
Cayley identity after \eqref{eq:car}, Lemma~\ref{lem:accretive-product}, and
preservation of spectra under similarity.  It does not depend on $n=2$.
This appendix records only the forward-inclusion estimates deferred there:
a polynomial approximation of the principal square root producing commuting
positive-stable rational factors, and the rationality of a common Lyapunov
metric so that one unimodular $U$ serves both factors on the closed disk.

Let $\mathcal R=\mathcal{RH}_\infty(\mathbb D)$ and $F^\sim(z)=F(1/\bar z)^*$
be as in Section~\ref{sec:cayley}.  Write $\mathscr S_{n,\mathrm{rat}}^\circ$
and $\mathscr A_{n,\mathrm{rat}}^\circ$ for the strict rational Schur and
strictly accretive classes of that section, identified by
\begin{equation}
P=(I_n+Q)(I_n-Q)^{-1},\qquad Q=(P-I_n)(P+I_n)^{-1}
\label{eq:supp-cayley}
\end{equation}
and
\begin{equation*}
P+P^*=2(I_n-Q^*)^{-1}(I_n-Q^*Q)(I_n-Q)^{-1}.
\end{equation*}

\begin{lem}[Commuting positive-stable factors]
\label{lem:supp-commuting-factor}
Suppose $A\in\mathcal R^{n\times n}$ and
$\sigma(A(z))\cap(-\infty,0]=\emptyset$ on the closed disk.  Then there
are commuting $X,Y\in\mathcal R^{n\times n}$ such that $A=XY$ and the
spectra of $X(z)$ and $Y(z)$ lie in the open right half-plane on the
closed disk.  The construction uses Mergelyan approximation of the
principal square root \cite{Rudin1987real,Higham2008functions}.
\end{lem}

\begin{proof}
The argument is the standard combination of a polynomial hull, Mergelyan
approximation of the principal square root
\cite{Rudin1987real,Higham2008functions}, and spectral mapping; it is
recorded under the disk conventions of the main text.

Let
$\mathcal K_A=\bigcup_{\lvert z\rvert\leq1}\sigma(A(z))$, a compact subset of
$\Omega=\mathbb C\setminus(-\infty,0]$.  In one complex variable the
polynomial hull fills only bounded complementary components.  Every point
of the cut may be joined to infinity along the cut, so the hull remains in
$\Omega$.  The hull is compact in the open set $\Omega$, so it stays at
positive distance from the cut.  Let $g(\lambda)=\sqrt{\lambda}$ be the
principal square root on $\Omega$.
The function is holomorphic on a neighborhood of $\widehat{\mathcal K}_A$,
and $\mathbb C\setminus\widehat{\mathcal K}_A$ is connected.  The principal
square root itself need not be rational \cite{Higham2008functions}; it is
used only as an approximation target.  By Mergelyan's theorem
\cite{Rudin1987real}, there are polynomials converging uniformly to $g$ on
$\widehat{\mathcal K}_A$.  Because $\mathcal K_A$ is invariant under
conjugation and $g(\bar\lambda)=\overline{g(\lambda)}$, the symmetrization
$(p+p^{\mathrm c})/2$ with $p^{\mathrm c}(\lambda)=\overline{p(\bar\lambda)}$
permits a real polynomial without enlarging the uniform error.  Put
\begin{equation*}
\begin{split}
\delta&=\min_{\lambda\in\mathcal K_A}\operatorname{Re}g(\lambda)>0,\\
m&=\min_{\lambda\in\mathcal K_A}\lvert g(\lambda)\rvert>0,
\qquad M=\max_{\lambda\in\mathcal K_A}\lvert g(\lambda)\rvert.
\end{split}
\end{equation*}
The three extrema are attained and strictly positive because $g$ is
continuous, $\mathcal K_A$ is compact, and $g(\Omega)$ lies in the open
right half-plane.  Choose a real polynomial $p$ so that
\begin{equation*}
\max_{\lambda\in\mathcal K_A}\lvert p(\lambda)-g(\lambda)\rvert
<\min\{\delta/2,m/2,\delta m/(4M)\}.
\end{equation*}
Then $\operatorname{Re}p(\lambda)>\delta/2$ and $\lvert p(\lambda)\rvert>m/2$.
Moreover
\begin{equation*}
\left\lvert\frac{\lambda}{p(\lambda)}-g(\lambda)\right\rvert
=\frac{\lvert g(\lambda)\rvert\,\lvert g(\lambda)-p(\lambda)\rvert}
{\lvert p(\lambda)\rvert}
<\delta/2,
\end{equation*}
so $\operatorname{Re}(\lambda/p(\lambda))>\delta/2$.  Set
\begin{equation*}
X=p(A),\qquad Y=X^{-1}A.
\end{equation*}
Polynomial spectral mapping \cite{Higham2008functions} shows that the
eigenvalues of $X(z)$ are the values $p(\lambda)$ for
$\lambda\in\sigma(A(z))$, hence lie in the open right half-plane and avoid
zero.  Thus $X$ is unimodular in $\mathcal R$.  Rational spectral mapping
gives the eigenvalues $\lambda/p(\lambda)$ of $Y(z)$, so both factors are
positive stable on the closed disk.  Since $X$ is a polynomial in $A$, the
factors commute and $XY=A$.  Diagonalizability of $A(z)$ is not used, and
the construction does not depend on $n=2$.
\end{proof}

\begin{lem}[Standard commuting metric]
\label{lem:supp-common-metric}
Let commuting $X,Y\in\mathcal R^{n\times n}$ be positive stable on the
closed disk.  There exists $U\in\mathrm{GL}_n(\mathcal R)$ such that
$UXU^{-1}$ and $UYU^{-1}$ are strictly accretive on the closed disk.
The construction combines the commuting constant-matrix Lyapunov method of
\cite{Narendra1994common}, applied fiberwise on the circle, with rational
spectral factorization \cite{Baggio2016factorization}.
\end{lem}

\begin{proof}
Fix a point on the unit circle, and suppress its argument.  Positive
stability gives unique positive-definite solutions of
\begin{equation}
Y^*R_Y+R_YY=I_n,\qquad X^*R_X+R_XX=I_n,
\qquad X^*H+HX=R_Y.
\label{eq:common-metric-equations}
\end{equation}
For example,
$R_Y=\int_0^\infty e^{-Y^*t}e^{-Yt}\,dt$ and
$H=\int_0^\infty e^{-X^*t}R_Ye^{-Xt}\,dt$.
The integrals converge and show strict positivity.  Let
$\mathcal L_X(Z)=X^*Z+ZX$ and $\mathcal L_Y(Z)=Y^*Z+ZY$.
Since $XY=YX$, these two linear operators commute.  Thus
\[
\mathcal L_X(\mathcal L_Y(H))
=\mathcal L_Y(\mathcal L_X(H))
=\mathcal L_Y(R_Y)=I_n.
\]
Uniqueness of the $X$-Lyapunov solution yields
$Y^*H+HY=R_X>0$.  Hence the same $H$ is a strict metric for both factors.

To prove rational dependence on the circle variable, replace the adjoints
in \eqref{eq:common-metric-equations} by paraconjugates.  For
$F=X$ or $Y$, vectorization of $F^\sim R+RF=G$ gives
\[
(I_n\otimes F^\sim+F^T\otimes I_n)\operatorname{vec}R
=\operatorname{vec}G.
\]
On the circle the eigenvalues of its coefficient matrix are sums of
eigenvalues of $F$ and their conjugates, with positive real parts.
The coefficient matrix therefore has a rational inverse with no circle
pole.  Successive application first to $R_Y$ and then to $H$ proves
that $H$ is real-rational and has no pole on the circle.  The pointwise
Hermitian solutions and rational identity imply $H^\sim=H$.
Continuity and compactness give $H(\zeta)\geq\varepsilon I_n$ on the
circle for some $\varepsilon>0$.

The rational spectral factorization theorem
\cite{Baggio2016factorization} now gives
\[
H=U^\sim U,\qquad U,U^{-1}\in\mathcal R^{n\times n}.
\]
Here $H$ has normal rank $n$, has no circle pole, and
$H\succeq\varepsilon I_n$ on the circle; these are the hypotheses that
permit the closed disk as a weakly unmixed-symplectic pole-and-zero
region (equivalently, invert the variable in the exterior convention).
The real-coefficient form of the theorem gives real-rational $U$.
Rationality then places $U$ and $U^{-1}$ in a neighborhood of the closed
disk, hence in $\mathcal R$.  Interior poles of $H$ are allowed: the metric
is used only on the circle.
On the circle,
\[
\begin{split}
UXU^{-1}+(UXU^{-1})^*
&=U^{-*}(X^*H+HX)U^{-1}>0,\\
UYU^{-1}+(UYU^{-1})^*
&=U^{-*}(Y^*H+HY)U^{-1}>0.
\end{split}
\]
Both similar matrices are analytic near the closed disk.  For a fixed unit
vector $v$, the real part of $v^*UXU^{-1}v$ is harmonic; its uniform
positive circle bound extends to the disk by the minimum principle.
The same argument applies to $Y$, proving the assertion.
\end{proof}

Neither lemma uses $n=2$: spectral mapping and vectorization are
dimension-independent.  Together they give the forward inclusion of
Theorem~\ref{thm:double-cayley}: Lemma~\ref{lem:supp-commuting-factor}
writes $A=XY$ with commuting positive-stable factors,
Lemma~\ref{lem:supp-common-metric} supplies one $U\in\mathrm{GL}_n(\mathcal R)$
making $P=UXU^{-1}$ and $Q=UYU^{-1}$ strictly accretive, and
\eqref{eq:supp-cayley} produces $Q_1,Q_2\in\mathscr S_{n,\mathrm{rat}}^\circ$
with
\begin{equation*}
\begin{split}
U^{-1}\mathcal C(Q_1)\mathcal C(Q_2)U
&=U^{-1}(UXU^{-1})(UYU^{-1})U\\
&=XY=A.
\end{split}
\end{equation*}

\section{Convolution form and truncated jet algebra}
\label{app:smith}

The node set \eqref{eq:nodes}, the local Smith chains
\eqref{eq:smith}--\eqref{eq:root-chain}, the cancellation identities
\eqref{eq:smith-hermite}, the infinity formula \eqref{eq:D-infty}, the jet
lengths $m_j$ and $m_\infty$, and the radial return \eqref{eq:radial-return}
are those of Section~\ref{sec:smith}; they are not re-proved.  Relative to
the running example $n=2$, the only changes of size are that
$\widehat{\mathcal M}$ is $2n\times 2n$ (hence $2n$ root chains rather than
four) and that the reconstruction $[\,I_n\;\widehat A\,]\widehat{\mathcal M}^{-1}$
is $n\times 2n$.  The argument that right multiplication by the invertible
germ $L_j$ neither creates nor cancels polar parts is unchanged.  This
appendix records only the truncated algebra that evaluates those identities
on a finite seed without first constructing global interpolants.

Expand $\nu_{j\ell}=\sum_{q}\nu_{j\ell,q}t^q$ and
$v_{j\ell}=\sum_{q}v_{j\ell,q}t^q$.  The cancellation conditions
\eqref{eq:smith-hermite} become the finite convolutions
\begin{equation}
[t^q]c_{j\ell}
=\nu_{j\ell,q}+\sum_{p=0}^q A_pv_{j\ell,q-p},
\label{eq:supp-c-conv}
\end{equation}
where $A_p$ are the normalized coefficients of $\widehat A$ at $\zeta_j$.
With the Smith data of $\widehat{\mathcal M}$ regarded as known, these are
affine-linear equations in the unknown jets of $\widehat A$.  They do not
constrain $\widehat A$ at other points.  If instead $B_p$ denotes the
interior coefficient of $B(w)$ at $\alpha_j=\zeta_j/r$, radial return
gives $A_p=r^{-p}B_p$, and the same equations read
\[
\nu_{j\ell,q}+\sum_{p=0}^q r^{-p}B_pv_{j\ell,q-p}=0.
\]
The coefficients $A_p$ in \eqref{eq:supp-c-conv} are already expressed
at the original node and are not scaled a second time.

\subsection{Truncated jet algebra}

Write
\begin{equation*}
\mathbf U(\eta)=\sum_{q=0}^{m-1}U_q\eta^q
\end{equation*}
for a recorded jet of length $m$ with $\det U_0\neq0$.  The inverse jet
$\sum_{q}V_q\eta^q$ is determined recursively by
\begin{equation}
\begin{split}
V_0&=U_0^{-1},\\
V_q&=-U_0^{-1}\sum_{p=1}^q U_pV_{q-p}
\qquad(q\geq1),
\end{split}
\label{eq:supp-inv-jet}
\end{equation}
so the coefficient of order $q$ depends only on input data of order at most
$q$.  Products are finite convolutions at the same order:
\begin{equation*}
[t^q](FG)=\sum_{a=0}^q F_aG_{q-a}.
\end{equation*}
The Cayley transform \eqref{eq:car} is a rational combination of one product
and one inverse, hence a finite operation on a truncated jet whenever
$I_n-Q_0$ is invertible.  In particular, if $Q=\gamma W$ with $0<\gamma<1$
and $W$ a Schur jet, then $I_n-Q_0$ is invertible.  Forming the seed
combination \eqref{eq:B-interior} in this truncated ring and returning
$\widehat A(z)=B(z/r)$ produces original-node coefficients $r^{-q}B_{j,q}$
that may be substituted directly into \eqref{eq:smith-hermite} and
\eqref{eq:D-infty}.  No global interpolant is required for that substitution.
The algebra does not depend on $n=2$: only the matrix size of each jet
changes.

\section{Unimodular Hermite lifting in every finite dimension}
\label{app:hermite}

This appendix proves Proposition~\ref{prop:hermite-lift}.  Notation is that
of Section~\ref{sec:smith}.  The scalar determinant lift below does not
depend on $n$.  The only additional step for $n\geq2$ is the classical
identity $\mathrm{SL}_n(B)=E_n(B)$ for commutative semilocal rings
\cite{Bass1968algebraic}, which is invoked, not re-proved.  When $n=1$ that
factor is absent: the scalar lift is already the matrix lift.

\begin{proof}[Proof of Proposition~\ref{prop:hermite-lift}]
Necessity is elementary.  If $U\in\mathrm{GL}_n(\mathcal R)$ matches the
jets, then $\det U$ is continuous, real-valued, and nonzero on the compact
interval $[-1,1]$, hence of constant sign.  At each real node the zeroth-order
jet is a real matrix, so $\det U_{\alpha,0}$ is real and nonzero, and all
such determinants must share that sign.

If the node set is empty, take $U=I_n$.  No interpolation constraint remains
and no quotient ring is needed.  The rest of the argument treats a nonempty
finite node set $\Lambda\subset\overline{\mathbb D}$, invariant under
conjugation, with conjugate-symmetric invertible jets
\begin{equation*}
\mathbf U_\alpha(t)=\sum_{q=0}^{m_\alpha-1}U_{\alpha,q}t^q,
\qquad \det U_{\alpha,0}\neq0.
\end{equation*}

Let
\begin{equation*}
p(z)=\prod_{\alpha\in\Lambda}(z-\alpha)^{m_\alpha}\in\mathbb R[z],
\qquad B=\mathbb R[z]/(p).
\end{equation*}
The Chinese remainder theorem decomposes $B$ into truncated real Taylor rings
at the real nodes and, for each nonreal conjugate pair, a truncated complex
Taylor ring viewed as a real algebra.  Two polynomial matrices realize the
same jets if and only if they are congruent modulo $p$, so the data determine
a class $\overline{U}\in B^{n\times n}$.  In each local factor an element is a unit
if and only if its constant residue is nonzero: writing $a=a_0(1+\varepsilon)$
with $\varepsilon$ nilpotent, the inverse is the finite sum
$a_0^{-1}\sum_{k\geq0}(-\varepsilon)^k$.  The hypothesis
$\det U_{\alpha,0}\neq0$ at every node therefore yields
$\overline{U}\in\mathrm{GL}_n(B)$ and
$\delta=\det\overline{U}\in B^\times$.

The ring $B$ is a finite-dimensional commutative real algebra, hence Artin,
and has only finitely many maximal ideals.  An element of $B$ is a unit if
and only if it lies in none of those ideals.

Let $\sigma\in\{1,-1\}$ be the common sign of $\det U_{\alpha,0}$ at the real
nodes, or $\sigma=1$ if there is no real node.  In each real local factor the
constant term of $\sigma\delta$ is positive, so one may write
$a_0(1+\varepsilon)$ with $a_0>0$ and $\varepsilon$ nilpotent, and define the
truncated logarithm
\begin{equation*}
\log\bigl(a_0(1+\varepsilon)\bigr)
=\log a_0+\sum_{k\geq1}\frac{(-1)^{k+1}}{k}\varepsilon^k.
\end{equation*}
Nilpotence makes the series finite.  In each nonreal local factor choose any
complex logarithm of the nonzero constant term and, on the conjugate factor,
the conjugate value; the nilpotent correction is again the same finite
series.  The Chinese remainder theorem supplies $\ell\in B$ with
$e^\ell=\sigma\delta$.

Choose a real polynomial representative $h$ of $\ell$ and set
$f=\sigma e^h$.  Then $f$ is entire and nowhere zero, and its jets modulo $p$
equal $\delta$.  Let $H$ be the Hermite polynomial representative of $\delta$,
of degree strictly less than $\deg p$.  The quotient
\begin{equation*}
g=\frac{f-H}{p}
\end{equation*}
has removable singularities at the nodes, hence extends to an entire
function, and is conjugate-symmetric.  Its Taylor polynomials at the origin
therefore have real coefficients and approximate $g$ uniformly on the closed
disk.  Choose a real polynomial $q$ so close to $g$ that
\begin{equation*}
\max_{\lvert z\rvert\leq1}\lvert p(z)\bigl(q(z)-g(z)\bigr)\rvert
<\min_{\lvert z\rvert\leq1}\lvert f(z)\rvert.
\end{equation*}
The minimum on the right is positive because $f$ is continuous and nowhere
zero on a compact set.  Put $d=H+pq$.  Then $d\equiv\delta\pmod p$ and
$\lvert d-f\rvert<\lvert f\rvert$ throughout the closed disk, so
\begin{equation*}
\lvert d\rvert\geq\lvert f\rvert-\lvert d-f\rvert>0
\qquad(\lvert z\rvert\leq1).
\end{equation*}
Thus $d$ has no zero on the closed disk.  (The same comparison on the unit
circle also yields the conclusion by Rouch\'e's theorem, together with the
strict inequality on the circle itself.)  In particular
$d,d^{-1}\in\mathcal R$.

When $n=1$, the matrix $\overline{U}$ is already the unit $\delta$, and the
scalar polynomial $d$ matches the prescribed (scalar) jets.  Set $U=d$.
Then $U\in\mathrm{GL}_1(\mathcal R)$, and the lift is complete.  No
$\mathrm{SL}_n$ factor is required.

Now assume $n\geq2$.  Set
\begin{equation*}
\overline{S}=\overline{U}\operatorname{diag}(\delta^{-1},1,\ldots,1)
\in\mathrm{SL}_n(B).
\end{equation*}
The identity $\mathrm{SL}_n(B)=E_n(B)$ for the commutative semilocal
(in fact Artin) ring $B$ is classical
\cite{Bass1968algebraic}.  Thus $\overline{S}$ is a finite product of
elementary matrices $E_{ij}(a)=I_n+a\,e_ie_j^T$.  Lift each elementary
parameter $a\in B$ to a real polynomial representative.
The corresponding product $S\in\mathrm{SL}_n(\mathbb R[z])$ reduces to
$\overline{S}$ modulo $p$.  No determinant error is introduced: each lifted
elementary matrix still has determinant $1$ exactly.  Set
\begin{equation*}
U=S\operatorname{diag}(d,1,\ldots,1).
\end{equation*}
Modulo $p$ one has
$U\equiv\overline{S}\operatorname{diag}(\delta,1,\ldots,1)=\overline{U}$, so
all prescribed jets are matched, including every higher-order coefficient
subject only to conjugate symmetry.  Moreover $\det U=d$ has no zero on the
closed disk and $U$ is a polynomial matrix, hence
$U,U^{-1}\in\mathcal R^{n\times n}$.

The two directions together are the stated criterion.  The only obstruction
is the real-node sign of the zeroth-order determinants; higher-order jets
create none.
\end{proof}

The truncated logarithms and the holomorphic approximation of $g$ are used
only to prove existence.  They are not steps of an exact algebraic
procedure.  The next lemma records a finite terminating construction over an
effectively presented real closed field.  The object enumerated is an
auxiliary scalar polynomial, not a controller McMillan degree: the jets are
already fixed, and no degree bound on $U$ is an input.

\begin{lem}[Effective Hermite lift]\label{lem:supp-hermite-effective}
Suppose the conjugate-symmetric node polynomial and the prescribed jets in
Proposition~\ref{prop:hermite-lift} are represented over an effectively
presented real closed field $\mathbb F\subset\mathbb R$, allowing finite
algebraic extensions for nonreal nodes.  If the determinant-sign condition
holds, a matching $U\in\mathrm{GL}_n(\mathbb F(z)\cap\mathcal R)$ can be
found by a finite terminating exact procedure.
\end{lem}

\begin{proof}
If the node set is empty, return $U=I_n$.  Otherwise the existence argument
of Proposition~\ref{prop:hermite-lift} produces a real polynomial $q_0$ of
some finite degree for which $d_0=H+pq_0$ has no zero in the closed disk.

Zero-freeness is open in the coefficients: on the compact disk,
$\lvert d_0\rvert$ attains a positive minimum, so every sufficiently close
polynomial of the same degree remains zero-free there.  Rational
coefficients are dense in that finite-dimensional real coefficient space.
Enumerate $q\in\mathbb Q[z]$ by degree and coefficient height.  For each
candidate, the assertion that $d=H+pq$ is zero-free on the closed disk is
the following first-order sentence over $\mathbb F$:
\begin{equation*}
\neg\,\exists x,y\quad
\left\{
\begin{array}{l}
x^2+y^2\leq1,\\
\operatorname{Re}d(x+iy)=0,\\
\operatorname{Im}d(x+iy)=0.
\end{array}
\right.
\end{equation*}
Real quantifier elimination, in the Tarski--Seidenberg and cylindrical
algebraic decomposition lineage
\cite{Tarski1951decision,Collins1975QE}, decides the sentence exactly.
Density together with the existence of $q_0$ guarantees that the enumeration
accepts a candidate after finitely many tests.  The enumeration is of the
scalar auxiliary polynomial $q$.  It is not a search over controller
McMillan degree, and it is invoked only after a feasible seed has already
been decided; termination uses the lift whose existence has already been
proved.

If $n=1$, set $U=d$ and stop.  If $n\geq2$, the remaining operations are
finite algebra over $\mathbb F$.  Factor $p$ over the real closed field,
form the Chinese remainder decomposition of $\mathbb F[z]/(p)$, test units
in its finitely many local factors, and factor
$\overline{S}=\overline{U}\operatorname{diag}(\delta^{-1},1,\ldots,1)$ into
elementary matrices by the same classical identity $\mathrm{SL}_n=E_n$
invoked in the proof of Proposition~\ref{prop:hermite-lift}.  Lifting the
resulting elementary parameters gives $S\in\mathrm{SL}_n(\mathbb F[z])$.
Then $U=S\operatorname{diag}(d,1,\ldots,1)$ is the required exact output.
Since the accepted $d$ is zero-free on the closed disk, both $U$ and
$U^{-1}$ belong to $\mathcal R^{n\times n}$.
\end{proof}

\section{Confluent Redheffer interpolation in the present conventions}
\label{app:redheffer}

This is the classical one-sided confluent Nevanlinna--Pick lift on the
disk \cite{BallGohbergRodman1990interpolation}, not a new interpolation
theorem.  The bitangential operator-argument theorem of
\cite[Thm.~1.2]{Ball2018bitangential} is the right half-plane counterpart;
the displayed disk kernel below is not identified with that statement.
Relative to $n=2$, the only change of size is the confluent state dimension
$d=n\sum_j m_j$ and the Kronecker factor $I_n$ in $T$.  Take distinct nodes
$\alpha_j\in\mathbb D$ of multiplicities $m_j$ and target jets
$W_{j,q}\in\mathbb C^{n\times n}$.  The left-sided full-jet data are
$T=\bigoplus_j\bigl(J_{m_j}(\alpha_j)\otimes I_n\bigr)$ and
$X=\operatorname{col}_j(e_1\otimes I_n)$, with $Y$ the stacked coefficients
$W_{j,0},\ldots,W_{j,m_j-1}$.  Here the Jordan block is explicitly
\[
J_m(\alpha)=
\begin{bmatrix}
\alpha&0&\cdots&0\\
1&\alpha&\ddots&\vdots\\
0&\ddots&\ddots&0\\
0&\cdots&1&\alpha
\end{bmatrix},
\qquad J_1(\alpha)=[\alpha].
\]
The ones are on the subdiagonal.  This convention, together with $e_1$
and the increasing Taylor-coefficient order in $Y$, is the one-sided
Lagrange--Sylvester stacking of
\cite[Ch.~16]{BallGohbergRodman1990interpolation}; it encodes the
prescribed jets and makes $(T,X)$ controllable.  Pairing conjugate Jordan
blocks by a fixed real similarity produces real $T,X,Y$ when the data are
conjugate-symmetric.

Let $\Pi$ be the unique solution of the Stein equation in the disk
realization of \cite{BallGohbergRodman1990interpolation},
\begin{equation*}
\Pi-T\Pi T^*=XX^*-YY^*.
\end{equation*}
Assume $\Pi>0$ and set
\begin{equation*}
\begin{split}
\Theta(z)={}&I_{2n}-(1-z)
\begin{bmatrix}X^*\\Y^*\end{bmatrix}(I_d-zT^*)^{-1}\\
&\times \Pi^{-1}(I_d-T)^{-1}[\,X\;-Y\,].
\end{split}
\end{equation*}
Partition $\Theta=[\Theta_{ij}]_{i,j=1}^2$.  The matrix is $J$-inner for
$J=\operatorname{diag}(I_n,-I_n)$ by the disk kernel identity of
\cite{BallGohbergRodman1990interpolation}.  In particular $\Theta_{22}$ is
unimodular over $\mathcal R$, and $\Theta_{21}R+\Theta_{22}$ remains
unimodular for every rational Schur parameter $R$.

\begin{lem}[Standard confluent Redheffer lift]\label{lem:supp-redheffer}
If $\Pi>0$, then as in the disk interpolation theory of
\cite{BallGohbergRodman1990interpolation}, all real-rational Schur
functions matching the prescribed confluent $n\times n$ jets are
\begin{equation*}
W=(\Theta_{11}R+\Theta_{12})
\bigl(\Theta_{21}R+\Theta_{22}\bigr)^{-1},
\qquad \lVert R\rVert_\infty\leq1.
\end{equation*}
The Redheffer denominator is unimodular over $\mathcal R$.  If the data are
conjugate-symmetric, $\Theta$ may be taken real-rational, and a
real-rational interpolant has a real-rational inverse parameter.
\end{lem}

\begin{proof}
The interpolation conditions and the Schur bound are those of the disk
theory in \cite{BallGohbergRodman1990interpolation}; they are not
re-proved.  The remaining points are only the present conventions.  The choice $R=0$
produces the explicit interpolant $W=\Theta_{12}\Theta_{22}^{-1}$.  On the
open set where the rational matrices are nonsingular, the inverse parameter
\begin{equation*}
\begin{bmatrix}H_1\\H_2\end{bmatrix}
=\Theta^{-1}\begin{bmatrix}W\\I_n\end{bmatrix},\qquad
R=H_1H_2^{-1}
\end{equation*}
is rational; the cited theorem makes $R$ Schur, so apparent poles in the
closed disk are removable.  A real state basis makes $\Theta$ real and the
inverse transform is unique, so a real interpolant has a real $R$.
\end{proof}

The admissible seed records two independent families of $n\times n$ jets.
Lemma~\ref{lem:supp-redheffer} is applied separately to each family.  The
resulting coefficient matrices are the blocks $\Theta_{k,\xi}$ in the atlas.

\section{Identity-jet covering}
\label{app:atlas}

This appendix proves the covering \eqref{eq:E-atlas}.  The polynomials
$b$ and $\beta$, the identity-jet writing $E-I_n=bG$, and the enumerated
centers $E_\nu=I_n+bG_\nu$ are those of Section~\ref{sec:atlas}; they are
not re-derived.  Relative to $n=2$, only the matrix size of $E$, $G$, and
$R_0$ changes.  For later reference,
\begin{equation*}
b(z)=\prod_j(z-\alpha_j)^{m_j},
\end{equation*}
\begin{equation*}
\beta(z)=\eta\prod_j
\left(\frac{z-\alpha_j}{1-\bar\alpha_jz}\right)^{m_j},
\qquad \lvert\eta\rvert=1,
\end{equation*}
with $\eta$ chosen so that $\beta$ has real coefficients, and
\begin{equation}
E_\nu=I_n+bG_\nu\in\mathrm{GL}_n(\mathcal R)
\label{eq:atlas-centers}
\end{equation}
as in Section~\ref{sec:atlas}.

\subsection{A standard inverse Cayley chart}

\begin{lem}[Standard inverse Cayley]\label{lem:std-inverse-cayley}
If $F\in\mathrm{GL}_n(\mathcal R)$ has the prescribed identity jets and
$\lVert F-I_n\rVert_\infty<1$, then there is a unique
$R\in\mathcal R^{n\times n}$ with $\lVert R\rVert_\infty<1$ such that
\begin{equation*}
F=\mathcal C(\beta R).
\end{equation*}
\end{lem}

\begin{proof}
The bound makes $F$ strictly accretive, so the inverse Cayley map of
Section~\ref{sec:cayley} yields a unique strict Schur $Q$.  Identity jets
allow division by $\beta$, and $\lvert\beta\rvert=1$ on the circle keeps
$\lVert R\rVert_\infty<1$.  This is the chart from that identity, not a new
Cayley theorem.
\end{proof}

The bound $\lVert F-I_n\rVert_\infty<1$ is a convenient neighborhood inside
the strictly accretive identity-jet matrices, not a necessary restriction on
the image of $R\mapsto\mathcal C(\beta R)$.  Directly, if $F$ is strictly
accretive and has identity jets, the same division $R=\beta^{-1}Q$ applies
without a small-norm hypothesis; conversely, $\beta R$ is strict Schur
whenever $\lVert R\rVert_\infty<1$, so $\mathcal C(\beta R)$ is strictly
accretive and, by the elementary identity
\begin{equation}
\mathcal C(\beta R)-I_n
=2\beta R\bigl(I_n-\beta R\bigr)^{-1},
\label{eq:atlas-cayley-id}
\end{equation}
has the prescribed identity jets.  The atlas formula
\eqref{eq:E-atlas} therefore imposes no extra small-norm constraint on its
outputs.  The neighborhood is used only to select a center in the next
lemma.

\subsection{Density of enumerated centers}

\begin{lem}[Standard density]\label{lem:atlas-density}
Let $E\in\mathrm{GL}_n(\mathcal R)$ have the prescribed identity jets and
write $E=I_n+bG$.  There is an index $\nu$ in the enumeration
\eqref{eq:atlas-centers} such that
\begin{equation}
\lVert E_\nu^{-1}E-I_n\rVert_\infty<1.
\label{eq:atlas-margin}
\end{equation}
\end{lem}

\begin{proof}
Approximate the finitely many coefficients of $G=N/a$ in the dense field
$\mathbb F$.  Unimodularity is open in $H^\infty$, so a close $E_\nu$
remains in $\mathrm{GL}_n(\mathcal R)$ and satisfies
\eqref{eq:atlas-margin}.  This is the standard density of
$\mathbb F$-rational coefficients together with openness of invertibility;
it is not a named interpolation theorem.
\end{proof}

\subsection{The identity-jet covering}

The two lemmas yield the covering \eqref{eq:E-atlas}.

\begin{proof}
Every matrix on the right-hand side of \eqref{eq:E-atlas} has identity jets
and a stable inverse.  Indeed, $\lVert\beta R_0\rVert_\infty
=\lVert R_0\rVert_\infty<1$, so $\mathcal C(\beta R_0)$ is a strictly
accretive element of $\mathrm{GL}_n(\mathcal R)$.  The identity
\eqref{eq:atlas-cayley-id} exhibits a factor $\beta$, hence
$\mathcal C(\beta R_0)$ has identity jets.  Multiplication by a center
$E_\nu$ from \eqref{eq:atlas-centers} preserves both unimodularity and
identity jets.

Conversely, take an arbitrary identity-jet matrix $E=I_n+bG$ in
$\mathrm{GL}_n(\mathcal R)$.  Lemma~\ref{lem:atlas-density} supplies an
enumerated unimodular center $E_\nu$ satisfying \eqref{eq:atlas-margin}.
The ratio $F=E_\nu^{-1}E$ still has identity jets: both factors match
$I_n$ in the truncated jet ring, so the quotient does as well.  Combined
with \eqref{eq:atlas-margin}, Lemma~\ref{lem:std-inverse-cayley} yields a
unique strict real-rational Schur parameter $R_0$ with
$F=\mathcal C(\beta R_0)$.  Thus $E=E_\nu\mathcal C(\beta R_0)$, which is
\eqref{eq:E-atlas}.
\end{proof}

The small-norm condition is used only to place $F$ inside the Cayley chart.
Every strict real-rational Schur parameter $R_0$ is admissible on the
right-hand side of \eqref{eq:E-atlas}.  Surjectivity of the full atlas in
Theorem~\ref{thm:atlas} uses this covering as its identity-jet half; the
remainder of that argument is in the main text.